\documentclass[11pt,oneside,dvipsnames]{amsart}
\usepackage{stackrel}
\usepackage[french, english]{babel}
\usepackage{a4wide}
\usepackage[utf8,latin1]{inputenc}
\usepackage{times}
\usepackage{adjustbox}
\usepackage{float}
\usepackage[cyr]{aeguill}
\usepackage{amsmath,amscd}
\usepackage{amssymb}
\usepackage{mathrsfs}
\usepackage{amsthm}
\usepackage{geometry}
\usepackage{graphicx}
\usepackage{epstopdf}
\usepackage{amsfonts}
\usepackage{amsmath}
\usepackage{amsmath,mathtools}
\usepackage{amsmath,graphicx}
\usepackage{amscd}
\usepackage{dsfont}
\usepackage{latexsym}
\usepackage{verbatim}
\usepackage{comment}
\usepackage{tikz-cd}
\usetikzlibrary{matrix,arrows}
\usepackage{tikz}
\usetikzlibrary{tikzmark}
\usepackage{tensor}
\usepackage{leftidx}
\usepackage[retainorgcmds]{IEEEtrantools}
\usepackage[title]{appendix}
\usepackage[mathscr]{euscript}
\usepackage[colorinlistoftodos]{todonotes}
\usepackage{bbm}
\usepackage{url}
\usepackage{upgreek}
\usepackage{mathtools, nccmath}
\usepackage[utf8]{inputenc}
\usepackage{multirow}
\usepackage{rotating}
\usepackage{marginnote}
\usepackage{hyperref} 
\usepackage{nameref}

\usepackage[all]{xy}
\usepackage[T1]{fontenc}

\usepackage{hyperref}
\usepackage{array}
\usepackage{tabularx}
\usepackage{yfonts}
\usepackage{setspace}  
\usepackage{geometry}
\makeatletter
\renewcommand\subsubsection{\@startsection{subsubsection}{3}{\z@}%
                                     {-3.25ex\@plus -1ex \@minus -.2ex}%
                                     {-0.5em}% <---changed
                                     {\normalfont\normalsize\bfseries}}
\makeatother

\theoremstyle{plain}
\newtheorem{Thm}{Theorem}
\newtheorem{Prop}[Thm]{Proposition}
\newtheorem{Lem}[Thm]{Lemma}
\newtheorem{Cor}[Thm]{Corollary}
\newtheorem{Con}[Thm]{Conjecture}

\theoremstyle{definition}

\newtheorem{Rem}[Thm]{Remark}

\newtheorem{Eg}[Thm]{Example}

\newtheoremstyle{named}{}{}{}{}{\bfseries}{.}{.5em}{\thmnote{#3}#1}
\theoremstyle{named}

\DeclareMathOperator{\ad}{ad}

\DeclareMathOperator{\Ch}{Ch}
\DeclareMathOperator{\sCh}{\mathbf{Ch}}

\DeclareMathOperator{\diag}{diag}

\DeclareMathOperator{\Gal}{Gal}

\DeclareMathOperator{\GL}{GL}

\DeclareMathOperator{\Id}{Id}

\DeclareMathOperator{\Irr}{Irr}

\DeclareMathOperator{\lisom}{\!\!\smash{\begin{array}{c}\sim\\[-1em]
\longrightarrow\end{array}}\!\!}

\DeclareMathOperator{\kk}{\mathds{k}}
\DeclareMathOperator{\ladic}{\bar{\mathbb{Q}}_{\ell}}

\DeclareMathOperator{\reg}{reg}

\DeclareMathOperator{\Rep}{Rep}
\DeclareMathOperator{\sRep}{\mathbf{Rep}}

\DeclareMathOperator{\SL}{SL}

\DeclareMathOperator{\Spec}{Spec}

\DeclareMathOperator{\SymF}{\mathbf{Sym}}
\DeclareMathOperator{\tr}{Tr}

\DeclareMathOperator{\X}{\mathbf{x}}

\DeclareMathOperator{\Z}{\mathbf{z}}

\DeclareMathOperator{\lb}{\!<\!}
\DeclareMathOperator{\rb}{\!>\!}
\DeclareMathOperator{\ds}{\!/\mkern-2mu/\mkern-2mu}

\newcommand*\circled[1]{
  \tikz[baseline=(char.base)]\node[shape=circle,draw,inner sep=0.2pt,font=\tiny,minimum size=8pt] (char) {#1};}

\makeatletter
\newcommand*{\rom}[1]{\expandafter\@slowromancap\romannumeral #1@}
\makeatother

\renewcommand{\descriptionlabel}[1]{\hspace\labelsep\upshape\bfseries #1.}
\makeatletter
\let\orgdescriptionlabel\descriptionlabel
\renewcommand*{\descriptionlabel}[1]{%
  \let\orglabel\label
  \let\label\@gobble
  \phantomsection
  \edef\@currentlabel{#1}%
  \let\label\orglabel
  \orgdescriptionlabel{#1}%
}
\makeatother

\usepackage{subfig}	 % \chaptionof

\title{E-Polynomials of Generic $\mathbf{\GL_n\rtimes\lb\sigma\rb}~$-Character Varieties: Unbranched Case}
\author{Cheng Shu}
\usepackage{pxfonts}
\usepackage{indentfirst}
\usepackage{setspace}  
\usepackage{hyperref}
\colorlet{ivory}{Apricot!30!}
\colorlet{space}{black!85!}
\definecolor{bgc}{RGB}{29, 44, 46}
\definecolor{txt}{RGB}{223, 222, 189}
\definecolor{cmd}{RGB}{206, 151, 88}
\begin{document}
\let\bs\boldsymbol
\maketitle
\begin{abstract}
For any unbranched double covering of compact Riemann surfaces, we study the associated character varieties that are unitary in the global sense, which we call $\GL_n\rtimes\lb\sigma\rb~$-character varieties. We introduce $k>0$ punctures on the surface, and restrict the monodromies around the punctures to generic semi-simple conjugacy classes in $\GL_n$, and compute the E-polynomials of these character varieties using the character table of $\GL_n(q)$. The result is expressed as the inner product of certain symmetric functions. We are then led to a conjectural formula for the mixed Hodge polynomial, which is built out of (modified) Macdonald polynomials, their self-pairings, and self-pairings of wreath Macdonald polynomials. 
\end{abstract}

\tableofcontents
\addtocontents{toc}{\protect\setcounter{tocdepth}{-1}}
%%% Usually after Introduction %%%
\setcounter{tocdepth}{1}
\numberwithin{Thm}{section}
\numberwithin{equation}{subsection}
\addtocontents{toc}{\protect\setcounter{tocdepth}{1}}
%%%%%%%%%%%%%%%%%%%
\section{Introduction}
In this article, we will call Macdonald polynomials what are usually called modified Macdonald polynomials in the literature, since only this version of Macdonald polynomial appears.
\subsection{Background}\hfill

Given a compact Riemann surface and a reductive algebraic group $G$, one can define three moduli spaces: the moduli space of Higgs $G$-bundles, the moduli space of flat $G$-connections and the $G$-character variety. The cohomologies of these spaces are identified via the non abelian Hodge theory and the Riemann-Hilbert correspondence. Early study of these cohomology groups focused on the computation of Betti numbers. The method was to exploit the $\mathbb{C}^{\ast}$-action on the moduli of Higgs bundles. See for example \cite{Hit} and \cite{Got}. However, such a computation is only manageable when the group has small ranks. In general, counting points over finite fields proved to be a much more efficient way to compute the cohomology. 

The project of counting points of character varieties over finite fields was initiated by Hausel and Rodriguez-Villegas in \cite{HRV}. For character varieties, the counting does not produce the Poincar\'e polynomial. Instead, one obtains the E-polynomial of the character variety over complex numbers, according to a theorem of Katz \cite[Appendix]{HRV}. The E-polynomial has as coefficients the alternating sum of the dimensions of the graded pieces of the weight filtration. The computation relies on our knowledge of the character table of the finite group of Lie type $\GL_n(q)$. With the E-polynomial obtained, in a subsequent work joint with Letellier, they proposed a conjectural formula for the mixed Hodge polynomial. This polynomial encodes the dimension of each graded piece of the weight filtration in each cohomological degree. Remarkably, the formula is dominated by (modified) Macdonald polynomials and their self-pairings. Recently, by counting Higgs bundles over finite fields, Schiffmann and Mellit have obtained a closed formula for the Poincar\'e polynomial of the moduli space of (parabolic) Higgs bundles with $G=\GL_n$ and arbitrary $n$, thus proving the conjecture of Hausel-Letellier-Rodriguez-Villegas at the level of Poincar\'e polynomials. See \cite{Sch}, \cite{Me3} and \cite{Me4}. Some other related conjectures have also been proved by Mellit. See \cite{Me1} and \cite{Me2}.

\subsection{$\GL_n\lb\sigma\rb~$-Character Varieties}\label{subsec-Char-Var}\hfill

In this article, we study character varieties that are unitary in the global sense. Let $\Sigma'$ be a compact Riemann surface of genus $g$, and let $p':\tilde{\Sigma}'\rightarrow\Sigma'$ be an unbranched double covering. We will assume $g>0$ throughout, so that such a covering exists. We introduce $k>0$ punctures on $\Sigma'$, resulting in the non compact surface $\Sigma$. Denote by $p:\tilde{\Sigma}\rightarrow\Sigma$ the restriction of $p'$ to $\Sigma$. Let $\mathcal{C}=(C_j)_{1\le j\le k}$ be a tuple of semi-simple conjugacy classes in $\GL_n$.  Finally, let $\sigma$ be an order 2 exterior automorphism of $\GL_n$, and denote by $\GL_n\lb\sigma\rb$ the semi-direct product $\GL_n\rtimes\lb\sigma\rb$. Our $\GL_n\lb\sigma\rb~$-character variety is defined by
$$
\Ch_{\mathcal{C}}(\Sigma):=\left\{(A_i,B_i)_i(X_j)_j\in \GL_n^{2g}\times \prod_{j=1}^{k} C_j\mid A_1\sigma(B_1)A_1^{-1}B_1^{-1}\prod_{i=2}^g[A_i,B_i]\prod_{j=1}^{k}X_j=1\right\}\ds\GL_n,
$$ where the bracket means the commutator. Note that $A_1\sigma(B_1)A_1^{-1}B_1^{-1}$ is just $[A_1\sigma,B_1]$. Also note that the defining equation is different from the branched case \cite[\S 1.3]{Shu3}. This is the moduli space of the homomorphisms $\rho$ that fit into the following commutative diagram:
$$
\begin{tikzcd}[row sep=2.5em, column sep=0.2em]
\pi_1(\Sigma) \arrow[dr, swap, "q_1"] \arrow[rr, "\rho"] && \GL_n\lb\sigma\rb \arrow[dl, "q_2"]\\
& \Gal(\tilde{\Sigma}/\Sigma)
\end{tikzcd}
$$
where $\Gal(\tilde{\Sigma}/\Sigma)\cong\mathbb{Z}/2\mathbb{Z}$ is the group of covering transformations, $q_1$ is the quotient by $\pi_1(\tilde{\Sigma})$, and $q_2$ is the quotient by the identity component. The monodromy of $\rho$ at the punctures must lie in the corresponding conjugacy classes $(C_j)_{1\le j\le k}$.

The character variety $\Ch_{\mathcal{C}}(\Sigma)$, via the non-abelian Hodge theory,  is diffeomorphic to the moduli space of (parabolic) unitary Higgs bundles, that is, torsors under the unitary group scheme over $\Sigma$ equipped with a Higgs field. The cohomology of this kind of moduli spaces is interesting in representation theory. See for example the work of Laumon and Ng\^o \cite{LN}. 

In a previous article \cite{Shu3}, the author has studied the case of branched coverings. Using the character table of $\GL_n(q)\lb\sigma\rb$, a formula for the E-polynomial of the character variety has been obtained. Surprisingly, the formula seems to suggest that the mixed Hodge polynomial is dominated by Macdonald polynomials and wreath Macdonlad polynomials (See below). Roughly speaking, a wreath Macdonald polynomial appears precisely where the homomorphism $\rho$ enters the twisted component $\GL_n\sigma$. In the current situation, a generator of $\pi_1(\Sigma)$ associated to a genus is mapped into $\GL_n\sigma$. We will see that this generator should give rise to a self-pairing of wreath Macdonald polynomial.

\subsection{Macdonald Polynomials and Wreath Macdonald Polynomials}\hfill

We write $\SymF[\Z]$ for the ring of symmetric functions in the variables $\Z=(z_1,z_2,\ldots)$ over the base field $\mathbb{Q}(z,w)$, the function field in $z$ and $w$. The ring $\SymF[\Z]$ is equipped with an inner product $\langle-,-\rangle$, called the Hall inner product. It can be deformed into another inner product $\langle-,-\rangle_{z,w}$. We denote by $\mathcal{P}$ the set of all partitions, including the empty one. For any $\lambda\in\mathcal{P}$, denote by $H_{\lambda}(\Z;z,w)$ the corresponding Macdonald polynomial. Denote by $$N_{\lambda}(z,w)=\langle H_{\lambda}(\Z;z,w),H_{\lambda}(\Z;z,w)\rangle_{z,w}$$the self-pairing of a Macdonald polynomial under the deformed inner product. It has an explicit expression: (See \cite[Corollary 5.4.8]{Hai}.)
\begin{equation}\label{Intro-N(q,t)}
N_{\lambda}(z,w)=\prod_{x\in\lambda}(z^{a(x)+1}-w^{l(x)})(z^{a(x)}-w^{l(x)+1}),
\end{equation}
where $x$ runs over the boxes in the Young diagram of $\lambda$, $a(x)$ and $l(x)$ denote the arm-length and leg-length respectively. We will also need a deformation of it: \begin{equation}
N_{\lambda}(u,z,w)=\prod_{x\in\lambda}(z^{a(x)+1}-uw^{l(x)})(z^{a(x)}-u^{-1}w^{l(x)+1}).
\end{equation}

Wreath Macdonald polynomials were introduced by Haiman in \cite{Hai}, whose existence was proved by Bezrukavnikov and Finkelberg in \cite{BF}. These are a family of (tensor) symmetric functions depending on an integer $r>1$ and an $r$-core. An $r$-core is simply a partition $\lambda$ such that there is no box $x\in\lambda$ with $h(x)\equiv0\!\!\!\mod r$, where $$h(x)=a(x)+l(x)+1$$ denotes the \textit{hook length} of $x$. In this article, we set $r=2$. The wreath Macdonald polynomials with a given 2-core are  parametrised by $\mathcal{P}^2=\mathcal{P}\times\mathcal{P}$, and the set of all wreath Macdonald polynomials, with varying 2-cores, is in bijection with $\mathcal{P}$. For any $\lambda\in\mathcal{P}$, write
\begin{equation}
\tilde{N}_{\lambda}(z,w)=\!\!\!\prod_{\substack{x\in\lambda\\h(x)\equiv0\!\!\!\mod 2}}\!\!\!(z^{a(x)+1}-w^{l(x)})(z^{a(x)}-w^{l(x)+1}),
\end{equation}
and 
\begin{equation}
\tilde{N}_{\lambda}(u,z,w)=\!\!\!\prod_{\substack{x\in\lambda\\h(x)\equiv0\!\!\!\mod 2}}\!\!\!(z^{a(x)+1}-uw^{l(x)})(z^{a(x)}-u^{-1}w^{l(x)+1}).
\end{equation}
Up to a "wreath-$\nabla$" term (see \cite[\S 8.2.2]{Shu3}) that plays no role in what follows, these are the self-pairing, and the deformed self-pairing, of the wreath Macdonald polynomial corresponding to $\lambda$.
 
\subsection{The Generating Series}\hfill

We define two generating series, which only depend on $g$ and $k$:
\begingroup
\allowdisplaybreaks
\begin{align*}
\Omega_{\bullet}(z,w):=&\sum_{\lambda\in\mathcal{P}}\frac{N_{\lambda}(zw,z^2,w^2)^{g-1}\tilde{N}_{\lambda}(zw,z^2,w^2)}{\tilde{N}_{\lambda}(z^2,w^2)}\prod^k_{j=1}H_{\lambda}(\Z_j;z^2,w^2)\\
\Omega_{\ast}(z,w):=&\sum_{\lambda\in\mathcal{P}}\frac{N_{\lambda}(zw,z^2,w^2)^{2g-1}}{N_{\lambda}(z^2,w^2)}\prod^k_{j=1}H_{\lambda}(\Z_j;z^2,w^2)^2,
\end{align*}
\endgroup
where $\{\mathbf{z}_1,\ldots,\mathbf{z}_k\}$ are independent variables.
\begin{Rem}
According to \cite{Shu3}, the "wreath-$\nabla$" term of the self-pairing of a wreath Macdonald polynomial should not be deformed. Therefore in $\Omega_{\bullet}(z,w)$, these terms cancel out. This is the reason why we omit them from the definition of $\tilde{N}_{\lambda}(z,w)$.
\end{Rem}

The reader can observe that the series $\Omega_{\bullet}(z,w)$ is roughly what we get from the series $\Omega_0(z,w)$ and $\Omega_1(z,w)$ defined in \cite{Shu3} by pretending that the number of ramification points is equal to 0. The essential difference is that there is no restriction on the 2-core in $\Omega_{\bullet}(z,w)$, and this is an obstruction to a uniform formula that works in both the branched case and the unbranched case. The origin of this difference lies in the representation theory of $\GL_n(q)\lb\sigma\rb$. As is explained in \cite{Shu3}, if a unipotent character of $\GL_n(q)$ corresponds to a partition with 2-core larger than $(1)$, then its extension to $\GL_n(q)\sigma$ must vanish on every semi-simple conjugacy class. In the unbranched case, the conjugacy classes are contained in $\GL_n(q)$, and so impose no restriction on the 2-core.

The series $\Omega_{\ast}(z,w)$ resembles the series defined by Hausel-Letellier-Rodriguez-Villegas for a Riemann surface of genus $2g-1$, but the Macdonald polynomial is squared.

\subsection{The Conjectural Mixed Hodge Polynomial}\hfill

For any tuple of semi-simple conjugacy classes $\mathcal{C}=(C_j)_{1\le j\le k}$, let $\bs\mu=(\mu_j)_{1\le j\le k}$ be the tuple of partitions of $n$ such that each $\mu_j$ is defined by the multiplicities of the eigenvalues of $C_j$. Let $\Ch_{\mathcal{C}}$ be the $\GL_n\lb\sigma\rb~$-character variety defined by $\mathcal{C}$. Denote by $d$ the dimension of $\Ch_{\mathcal{C}}$, which only depend on $\bs\mu$, $g$ and $n$. For any $\lambda\in\mathcal{P}$, denote by $h_{\lambda}(\Z)$ the corresponding complete symmetric function.
The summand in $\Omega_{\ast}$ corresponding to the empty partition is equal to 1, therefore we can apply the formal expansion $$\frac{1}{1+x}=\sum_{m\ge 0}(-1)^mx^m$$ to $1+x=\Omega_{\ast}$.
Define the rational functions in $z$ and $w$:
\begin{equation}\label{intro-Hmu}
\mathbb{H}_{\bs\mu}(z,w):=~\left\langle\frac{\Omega_{\bullet}(z,w)^2}{\Omega_{\ast}(z,w)},\prod^{k}_{j=1}h_{\mu_j}(\Z_j)\right\rangle.
\end{equation}
The definition resembles the branched case, but does not depend on the parity of $n$.

\begin{Con}\label{The-Conj}
Let $\mathcal{C}=(C_j)_{1\le j\le k}$ be a strongly generic tuple of semi-simple conjugacy classes in $\GL_n$. Then 
\begin{itemize}
\item[(i)] The rational function $\mathbb{H}_{\bs\mu}(z,w)$ is a polynomial. It has degree $d$ in each variable, and each monomial in it has even degree. Moreover, $\mathbb{H}_{\bs\mu}(-z,w)$ has non-negative integer coefficients.
\item[(ii)] The mixed Hodge polynomial $H_c(\Ch_{\mathcal{C}};x,y,t)$ is a polynomial in $q:=xy$ and $t$.
\item[(iii)] The mixed Hodge polynomial is given by the following formula:
\begin{equation*}
H_c(\Ch_{\mathcal{C}};q,t)=(t\sqrt{q})^d\mathbb{H}_{\bs\mu}(-t\sqrt{q},\frac{1}{\sqrt{q}}).
\end{equation*}
In particular, it only depends on $\bs\mu$, and not on the generic eigenvalues of $\mathcal{C}$.
\end{itemize}
\end{Con}
\begin{Rem}
Although $\mathcal{C}$ is a tuple of conjugacy classes of $\GL_n(q)$, the definition of genericity is very different from that in \cite{HLR}. (See \S\ \ref{ss-GCC}.)
\end{Rem}

The above conjectural formula, combined with some fundamental symmetries in Macdonald polynomials and wreath Macdonald polynomials, implies the following.
\begin{Con}(Curious Poincar\'e Duality)\label{Cur-Poin}
We have
\begin{equation*}
H_c(\Ch_{\mathcal{C}};\frac{1}{qt^2},t)=(qt)^{-d}H_c(\Ch_{\mathcal{C}};q,t).
\end{equation*}
\end{Con}

\subsection{Main Theorem and Evidences of the Conjecture}\label{Intro-Main-Thm}\hfill

Our main theorem is a formula for the E-polynomial.
\begin{Thm}\label{Main-Thm-intro}
The $t=-1$ specialisation of part (iii) of Conjecture \ref{The-Conj} holds:
$$
H_c(\Ch_{\mathcal{C}};q,-1)=(\sqrt{q})^d\mathbb{H}_{\mathbf{B}}(\sqrt{q},\frac{1}{\sqrt{q}}).
$$
\end{Thm}

Then the Curious Poincar\'e duality specialises to the following symmetry.
\begin{Thm}\label{E-Cur-Poin}
The $t=-1$ specialisation of Conjecture \ref{Cur-Poin} holds:
$$E(\Ch_{\mathcal{C}};q)=q^dE(\Ch_{\mathcal{C}};q^{-1}).$$
\end{Thm}

Our computation of the E-polynomial follows the line of \cite{Shu3}. However, there are major differences. On the one hand, in view of the defining equation of the character variety (see \S \ref{subsec-Char-Var}), the computation will only use the characters of $\GL_n(q)$, and not those of $\GL_n(q)\lb\sigma\rb$. The consequence is that, partitions are allowed to have arbitrary 2-cores, as we have seen. On the other hand, due to the presence of the exterior automorphism $\sigma$, only those irreducible characters of $\GL_n(q)$ that are $\sigma$-stable have non trivial contributions to the E-polynomial. This results in the expression (\ref{intro-Hmu}) which looks more like the series in \cite{Shu3} instead of the one in \cite{HLR}. Therefore, we will use a mixture of  techniques from \cite{HLR} and \cite{Shu3}.

The E-polynomial is a main evidence of our conjecture. However, as we can observe from the definition of $\Omega_{\bullet}(z,w)$, if we make the specialisations $z=\sqrt{q}$ and $w=\sqrt{q}^{-1}$, i.e. pass to the E-polynomial, then the wreath terms $\tilde{N}_{\lambda}$ become invisible. This is one reason for which the conjectural formula in the unbranched case is formulated after the branched case. In \cite{Shu3}, it is the representation theory of $\GL_n(q)\lb\sigma\rb$ that suggests the possible connection to wreath Macdonald polynomials. In the current circumstance, the wreath terms can only be seen when we explicitly compute the mixed Hodge polynomials for $n=1$ and $n=2$ (and arbitrary $g>0$). These computations are given at the end of the article. Without our previous experience with the branched case, we would have no idea how to explain these unexpected terms. 

To test our conjectural formula\footnote{The author uses MATLAB.} we have computed $\mathbb{H}_{\bs\mu}(z,w)$ when $n=3$, $g=1,2,3$ and,
\begin{itemize}
\item{$k=1$.} $\mathcal{C}=(C_1)$ with $C_1$ regular 
\item{$k=2$.} $\mathcal{C}=(C_1,C_2)$ with $C_1$ and $C_2$ regular;
\item{$k=2$.} $\mathcal{C}=(C_1,C_2)$ with $C_1$ regular and $C_2$ conjugate to $\diag(a,a,b)$, $a\ne b$.
\end{itemize}
There is no geometric result in the literature that allows us to compute the mixed Hodge polynomial of these character varieties when $n=3$. However, our computation shows that part (i) of Conjecture \ref{The-Conj} is true in every case we have tested. Since part (i) of the conjecture is a sufficient condition for the formula in part (iii) to be a polynomial in $q$ and $t$ with non negative integer coefficients, this indicates a geometric origin of $\mathbb{H}_{\bs\mu}(z,w)$. 

Let us give one example. We put $n=3$, $g=1$, $k=1$, and take one regular semi-simple conjugacy class. The complete symmetric function $h_1(\Z)^3$ corresponds to this conjugacy class. Unwinding the definitions, we have
\begingroup
\allowdisplaybreaks
\begin{align*}
\frac{\Omega_{\bullet}(z,w)^2}{\Omega_{\ast}(z,w)}=&\frac{\tilde{N}_{(2)}(zw,z^2,w^2)}{\tilde{N}_{(2)}(z^2,w^2)}H_{(1)}H_{(2)}+\frac{\tilde{N}_{(1^2)}(zw,z^2,w^2)}{\tilde{N}_{(1^2)}(z^2,w^2)}H_{(1)}H_{(1^2)}\\
&+\frac{\tilde{N}_{(3)}(zw,z^2,w^2)}{\tilde{N}_{(3)}(z^2,w^2)}H_{(3)}+\frac{\tilde{N}_{(1^3)}(zw,z^2,w^2)}{\tilde{N}_{(1^3)}(z^2,w^2)}H_{(1^3)}+H_{(21)}\\
&-\frac{N_{(1)}(zw,z^2,w^2)}{N_{(1)}(z^2,w^2)}H_{(1)}^3+\text{terms not of degree 3},
\end{align*}
\endgroup
where we have omitted the variables of the Macdonald polynomials. Note that the partition $(21)$ is a 2-core, and so $\tilde{N}_{(21)}(z,w)=1$. Consequently, $H_{(21)}$ has coefficient 1. For the same reason, we do not see the contribution of partition $(1)$ in the coefficients of $H_{(1)}H_{(2)}$ and $H_{(1)}H_{(1^2)}$. Taking the Hall inner product with $h_1(\Z)^3$ gives
\begingroup
\allowdisplaybreaks
\begin{align*}
\mathbb{H}_{\bs\mu}(z,w)=&~\frac{(z^3-w)^2}{(z^4-1)(z^2-w^2)}(3+3z^2)+\frac{(z-w^3)^2}{(1-w^4)(z^2-w^2)}(3+3w^2)\\
&+\frac{(z^3-w)^2}{(z^4-1)(z^2-w^2)}(1+2z^2(z^2+1)+z^6)\\
&+\frac{(z-w^3)^2}{(1-w^4)(z^2-w^2)}(1+2w^2(w^2+1)+w^6)\\
&+(1+2(z^2+w^2)+z^2w^2)-6\frac{(z-w)^2}{(z^2-1)(1-w^2)}\\
=&~w^6+w^4\,z^2+2\,w^4-2\,w^3\,z+w^2\,z^4+3\,w^2\,z^2+8\,w^2-2\,w\,z^3-4\,w\,z\\
&+z^6+2\,z^4+8\,z^2+5.
\end{align*}
\endgroup

\subsection*{Organisation of the Article.}\hfill

As in \cite{HLR} and \cite{Shu3}, there are three key ingredients in computing the E-polynomial, and they are collected in Section \ref{Section-Pre} and Section \ref{Section-GLn(q)}.
\begin{itemize}
\item[(1)] The Frobenius formula (\ref{eq-Frob-Form}) that reduces the point-counting problem to the evaluation of $\sigma$-stable irreducible characters of $\GL_n(q)$.
\item[(2)] A theorem of Lusztig and Srinivasan that expresses an irreducible character of $\GL_n(q)$ as a linear combination of Deligne-Lusztig characters $R^G_{T_w}\theta_{T_w}$ (\ref{eq-alm-char}).
\item[(3)] The character formula (\ref{eq-char-formula}) which reduces the computation of $R^G_{T_w}\theta_{T_w}$ to the Green function and the linear characters $\theta_{T_w}$.
\end{itemize}
The Green functions can be realised as transition matrices between symmetric functions (\ref{eq-Green}). The linear characters are computed in Section \ref{Sec-Lin}. In Section \ref{Section-Char-Var}, we give the definition of our character variety, as well as generic conjugacy classes, which is a key assumption in our main theorem. The computation of E-polynomial is achieved in Section \ref{Section-E-poly}, using results from previous sections. In the last section, we check that our conjectural formula is consistent with some known cohomological results when $n=1$ and $n=2$.

\subsection*{Acknowledgement.}\hfill

This article is prepared when I was unemployed at home. I thank my family for their support. I thank Tam\'as Hausel for answering a question.

\numberwithin{Thm}{subsection}
\numberwithin{equation}{subsubsection}
%%%%%%%%%%%%%%%%%%%%%%
\section{Preliminaries}\label{Section-Pre}

\subsection{General Notations}\hfill

Let $G$ be an algebraic group, and let $X$, $Y$ and $Z$ be subvarieties of $G$. We will write $N_X(Y)=\{x\in X\mid xYx^{-1}=Y\}$ and $N_X(Y,Z)=\{x\in X\mid xYx^{-1}=Y,~xZx^{-1}=Z\}$. If $g\in G$, we denote by $C_G(g)=\{h\in G\mid hgh^{-1}=g\}$ the centraliser of $g$ in $G$. The identity component of an algebraic group $G$ is denoted by $G^{\circ}$. If $\phi$ is an automorphism of a set $X$, we denote by $X^{\phi}$ the subset of the fixed points of $\phi$. This convention applies to the case where $X$ is the set of closed points of an algebraic variety and $\phi$ is an endomorphism of $X$. If $G$ is a finite group, then $\Irr(G)$ denotes the set of irreducible characters of $G$.

For any $n\in\mathbb{Z}_{>0}$, we denote by $\mathfrak{S}_n$ the symmetric group and by $\mathfrak{W}_n$ the semi-direct product $(\mathbb{Z}/2\mathbb{Z})^n\rtimes\mathfrak{S}_n$, where $\mathfrak{S}_n$ acts on $(\mathbb{Z}/2\mathbb{Z})^n$ by permuting the factors.

If $\sigma$ is an order 2 exterior automorphism of $\GL_n$, then we can define a semi-direct product $\GL_n\rtimes\lb\sigma\rb$. In this article, we use the automorphisms $\sigma$ defined in \cite[\S 3.1.1]{Shu3}.

\subsection{Algebraic groups defined over a finite field}

\subsubsection{}
We will denote by $p>2$ a prime number and $q$ a power of $p$. We denote by $\kk$ an algebraic closure of $\mathbb{F}$. An algebraic group $G$ over $\kk$ is defined over $\mathbb{F}_q$ if there is an algebraic group $G_0$ over $\mathbb{F}_q$ such that $G_0\otimes_{\mathbb{F}_q}\kk\cong G$. If this is the case, we denote by $F$ the geometric Frobenius endomorphism. We denote by $G^F$ the finite subgroup of $F$-fixed points. We may also denote it by $G(q)$. If $X\subset G$ is a subvariety, we say that $X$ is $F$-stable if $F(X)=X$.

\subsubsection{}
Let $G$ be a connected reductive group defined over $\mathbb{F}_q$, and let $T\subset G$ be an $F$-stable maximal torus. Denote by $W$ the Weyl group of $G$ defined by $T$. The $G^F$-conjugacy classes of $F$-stable maximal tori are parametrised by the $F$-conjugacy classes of $W$. Two elements $v$ and $w$ of $W$ are $F$-conjugate if there is some $x\in W$ such that $xvF(x)^{-1}=w$. This bijection is given as follows. Let $w\in W$ and choose $\dot{w}\in N_G(T)$ representing $w$. By Lang-Steinberg theorem, there exists $g\in G$ such that $g^{-1}F(g)=\dot{w}$. Then $gTg^{-1}$ is an $F$-stable maximal torus whose $G^F$-conjugacy class corresponds to $w$. Conversely, any $F$-stable maximal torus can be written as $gTg^{-1}$ for some $g\in G$. Since $g^{-1}F(g)$ normalises $T$, it defines an element of $W$ whose $F$-conjugacy class corresponds to the $G^F$-conjugacy class of $gTg^{-1}$.

Fix an $F$-stable Levi subgroup $L_0\subset G$. The $G^F$-conjugacy classes of the $G$-conjugates of $L_0$ are in bijection with the $F$-conjugacy classes of $N_G(L_0)/L_0$. This bijection is completely analogue to the torus case above. Let $w\in N_G(L_0)/L_0$, choose $\dot{w}\in N_G(L_0)$ representing it, and let $g\in G$ be such that $g^{-1}F(g)=\dot{w}$. Then $L:=gL_0g^{-1}$ is also an $F$-stable Levi subgroup. Define $F_{w}:=\ad\dot{w}\circ F$. It is another Frobenius endomorphism on $G$, and $L_0$ is $F_{w}$-stable. Now $\ad g$ is an isomorphism between $L_0$ and $L$. We have a commutative diagram:
\begin{equation}\label{F_w}
\begin{CD}
L_0 @>\ad g>> L\\
@VF_wVV @VVFV\\
L_0 @>>\ad g> L
\end{CD}
\end{equation}
In particular, $L_0^{F_w}\cong L^F$.

\subsection{Frobenius formula}
\subsubsection{}
Let $G$ be a finite group and let $\sigma$ be an automorphism of $G$ of order 2. Then $\sigma$ induces an action on $\Irr(G)$ by composition. Denote by $\Irr(G)^{\sigma}$ the subset of $\sigma$-fixed elements. 
\begin{Prop}\label{Frob-Form}
Let $\mathcal{C}=(C_j)_{1\le j\le k}$ be an arbitrary $k$-tuple of conjugacy classes in $G$. Then we have the following counting formula:
\begin{equation}\label{eq-Frob-Form}
\begin{split}
&\#\left\{(A_i,B_i)_i(X_j)_j\in G^{2g}\times \prod_{j=1}^{k}C_j\Big\arrowvert A_1\sigma(B_1)A_1^{-1}B_1^{-1}\prod_{i=2}^{g}[A_i,B_i]\prod_{j=1}^{k}X_j=1\right\}\\
=&|G|\sum_{\chi\in\Irr(G)^{\sigma}}\Big(\frac{|G|}{\chi(1)}\Big)^{2g-2}\prod_{i=1}^{k}
\frac{|C_i|\chi(C_i)}{\chi(1)}.
\end{split}
\end{equation}
\end{Prop}
\begin{proof}
If $f_1$ and $f_2$ are complex-valued functions on $G$ that are invariant under conjugation by $G$, then their convolution product is defined by $$(f_1\ast f_2)(x)=\sum_{yz=x}f_1(y)f_2(z),$$which is also invariant under conjugation. Define such a function $n^{1}$ by
$$
n^1(x)=\#\left\{(A,B)\in G^{2}\Big\arrowvert ABA^{-1}B^{-1}=x\right\},
$$
and another function $n'$ by
$$
n'(x)=\#\left\{(A,B)\in G^{2}\Big\arrowvert A\sigma(B)A^{-1}B^{-1}=x\right\}.
$$
Denote by $n^{g-1}$ the convolution product of $g-1$ copies of $n^1$. Denote by $1_{C_i}$ the characteristic function of the class $C_i$. Then
\begingroup
\allowdisplaybreaks
\begin{align*}
&\#\left\{(A_i,B_i)_i(X_j)_j\in G^{2g}\times \prod_{j=1}^{k}C_j\Big\arrowvert A_1\sigma(B_1)A_1^{-1}B_1^{-1}\prod_{i=2}^{g}[A_i,B_i]\prod_{j=1}^{k}X_j=1\right\}\\
=&(n'\ast n^{g-1}\ast 1_{C_1}\ast\cdots\ast 1_{C_{k}})(1).
\end{align*}
\endgroup
By \cite[(2.7.2.6), (2.7.2.7)]{Shu3}, we have
\begingroup
\allowdisplaybreaks
\begin{align}
&(n'\ast n^{g-1}\ast 1_{C_1}\ast\cdots\ast 1_{C_{k}})(1)\nonumber\\
=&\frac{1}{|G|}\sum_{\chi\in\Irr(G)}\chi(1)^2\mathcal{F}(n')(\chi)(\mathcal{F}(n^1)(\chi))^{g-1}\prod_{i=1}^{k}\frac{|C_i|\chi(C_i)}{\chi(1)},\label{brutal-Frob-For}
\end{align}
\endgroup
where for any function $f$ on $G$ and any character $\chi$,
$$
\mathcal{F}(f)(\chi):=\sum_{h\in G}\frac{f(h)\chi(h)}{\chi(1)}.
$$
According to \cite[Lemma 3.1.3]{HLR}, we have
$$
\mathcal{F}(n^1)(\chi)=
\big(\frac{|G|}{\chi(1)}\big)^2.
$$
It remains to compute $\mathcal{F}(n')(\chi)$.

Let $\rho:G\rightarrow \GL(V)$ be an irreducible representation with character $\chi$. We have
$$
\sum_{h\in G}\frac{n'(h)\chi(h)}{\chi(1)}=\sum_{(a,b)\in G^2}\frac{\chi(a\sigma(b)a^{-1}b^{-1})}{\chi(1)}=\frac{1}{\chi(1)}\tr\sum_{(a,b)\in G^2}\rho(a\sigma(b)a^{-1}b^{-1}).
$$
Note that $\sum_a\rho(a\sigma(b)a^{-1})$ is an endomorphism of $V$ that commutes with $\rho(h)$ for any $h\in G$. By Schur Lemma, it must be a scalar endomorphism. We deduce that
$$
\sum_a\rho(a\sigma(b)a^{-1})=\frac{|G|\cdot\chi\circ\sigma(b)}{\chi(1)}\Id_V.
$$
Using the orthogonality of irreducible characters, we compute
\begingroup
\allowdisplaybreaks
\begin{align*}
&\frac{1}{\chi(1)}\tr\sum_{(a,b)\in G^2}\rho(a\sigma(b)a^{-1}b^{-1})=\frac{1}{\chi(1)}\sum_{b\in G^2}\frac{|G|\cdot\chi\circ\sigma(b)}{\chi(1)}\chi(b^{-1})\\
=&
\begin{cases}
0 & \text{if } \chi\circ\sigma\ne\chi\\
\left(\frac{|G|}{\chi(1)}\right)^2 & \text{otherwise}.
\end{cases}
\end{align*}
\endgroup
We conclude that in (\ref{brutal-Frob-For}), only the $\sigma$-stable irreducible characters contribute to the sum.
\end{proof}

\subsection{Symmetric Functions}\hfill

Now the letter $q$ denotes an indeterminate.
\subsubsection{}
Let $\Z=(z_1,z_2,\ldots)$ be an infinite series of variables. The ring of symmetric functions over a field $\mathbb{K}$ is denoted by $\SymF_{\mathbb{K}}[\Z]$. In this article, $\mathbb{K}$ can be $\mathbb{Q}$, $\mathbb{Q}(q)$, or $\mathbb{Q}(q,t)$, i.e. functions fields in variables $q$, $t$ over rational numbers $\mathbb{Q}$, and in these cases, we will write $\SymF_q[\Z]$ and $\SymF_{q,t}[\Z]$. We will omit the subscript if there is no confusion with the base ring. For any partition $\lambda\in\mathcal{P}$, the corresponding Schur symmetric functions, monomial symmetric functions, complete symmetric functions and the power sum symmetric functions are denoted by $s_{\lambda}(\Z)$, $m_{\lambda}(\Z)$, $h_{\lambda}(\Z)$ and $p_{\lambda}(\Z)$ respectively.

For any partition $\lambda=(1^{m_1},2^{m_2},\ldots)$, define $z_{\lambda}=\prod_ii^{m_i}m_i!$. We have,
\begingroup
\allowdisplaybreaks
\begin{align*}
p_{\lambda}(\Z)=&\sum_{\tau\in\mathcal{P}}\chi^{\tau}_{\lambda}s_{\tau}(\Z),\\
s_{\lambda}(\Z)=&\sum_{\tau}z_{\tau}^{-1}\chi^{\lambda}_{\tau}p_{\tau}(\Z),
\end{align*}
\endgroup
where $\chi^{\tau}_{\lambda}$ is the value of the irreducible character of $\mathfrak{S}_n$ ($n=|\tau|=|\lambda|$) of class $\tau$ at a conjugacy class of class $\lambda$.

\subsubsection{}
There are two families of Haill-Littlewood symmetric functions defined in \cite{Mac}, and they are denoted by $P_{\lambda}(\Z;q)$ and $Q_{\lambda}(\Z;q)$, for $\lambda\in\mathcal{P}$. The function $b_{\lambda}(q)$ is the difference between them: $$b_{\lambda}(q)P_{\lambda}(\Z;q)=Q_{\lambda}(\Z;q).$$By \cite[(2.6), (2.7)]{Mac}, we have
\begin{equation}\label{1/|GLn|}
\frac{1}{|\GL_n(q)|}=q^{-2n((1^n))-n}b_{(1^n)}(q^{-1})^{-1}.
\end{equation} 
We have
\begin{equation}
s_{\lambda}(\Z)=\sum_{\tau}K_{\lambda,\tau}(q)P_{\tau}(\Z;q),
\end{equation}
for some polynomials $K_{\lambda,\tau}(q)$ in $q$ and they are called Kostka-Foulkes polynomials. The modified Kostka-Foulkes polynomials are defined by
\begin{equation}
\tilde{K}_{\lambda,\tau}(q)=q^{n(\tau)}K_{\lambda,\tau}(q^{-1}).
\end{equation}
The modified Hall-Littlewood function is defined by
\begin{equation}
\tilde{H}_{\lambda}(\Z;q)=\sum_{\tau}\tilde{K}_{\tau,\lambda}(q)s_{\tau}(\Z).
\end{equation}
Finally, the Green polynomial is defined by 
\begin{equation}\label{eq-Green}
\mathcal{Q}^{\tau}_{\lambda}(q)=\sum_{\nu}\chi^{\nu}_{\lambda}\tilde{K}_{\nu\tau}(q).
\end{equation}

\subsubsection{}
The ring $\SymF_{\mathbb{K}}[\Z]$ has a natural $\lambda$-ring structure $\{p_n\}_{n\in\mathbb{Z}_{>0}}$ with $p_n$ sending $p_1(\Z)$ to $p_n(\Z)$ and sending $q$ and $t$ to their $n$-th powers. Let $A$ be a $\lambda$-ring containing $\mathbb{K}$ as a $\lambda$-subring. Given any element $x$ of $A$, there is a unique $\lambda$-ring homomorphism $\varphi_x:\SymF_{\mathbb{K}}[\Z]\rightarrow A$ sending $p_1(\Z)$ to $x$. For any $F\in\SymF_{\mathbb{K}}[\Z]$, its image under this map is denoted by $F[x]$. Concretely, we can write $F$ as a polynomial $f(p_n(\Z)\mid n\in\mathbb{Z}_{>0})$, we have $F[x]=f(p_n(x)\mid n\in\mathbb{Z}_{>0})$. This is the plethystic substitution for $F$.

\subsubsection{}
The Hall inner product on $\SymF[\Z]$ is defined by $$\langle s_{\lambda}(\Z),s_{\mu}(\Z)\rangle=\delta_{\lambda,\mu},$$for any partitions $\lambda$ and $\mu$. We have $$\langle p_{\lambda}(\Z),p_{\mu}(\Z)\rangle=z_{\lambda}\delta_{\lambda,\mu}.$$The $q,t$-inner product on $\SymF_{q,t}[\Z]$ is defined by $$\langle F[\Z],G[\Z]\rangle_{q,t}:=\langle F[\Z],G[(q-1)(1-t)\Z]\rangle$$for any $F$, $G\in\SymF_{q,t}[\Z]$.

%%%%%%%%%%%%%%%%%%%%%%
\section{$\sigma$-Stable Irreducible Characters of $\GL_n(q)$}\label{Section-GLn(q)}
In this section, we recall the character formula for Deligne-Lusztig inductions, the construction of $\sigma$-stable irreducible characters of $\GL_n(q)$, and introduce some combinatorial data called types. Types will be used in our explicit computation of the E-polynomial in Section \ref{Section-E-poly}.

\subsection{Deligne-Lusztig Inductions}

\subsubsection{}
Let $H$ be a connected reductive group defined over $\mathbb{F}_q$ and denote by $F$ the geometric Frobenius endomorphism. If $L\subset H$ is an $F$-stable Levi subgroup, we denote by $R^H_L$ the Deligne-Lusztig induction, which is a $\ladic$-linear map from the set of $L^F$-invariant functions on $L^F$ to the set of $H^F$-invariant functions on $H^F$. We fix an $F$-stable maximal torus $T\subset H$, and denote by $W_H$ the Weyl group of $H$ defined by $T$. For any $w\in W_H$, we choose an $F$-stable maximal torus $T_w$ corresponding to the $F$-conjugacy class of $w$. Let $\mathbf{1}$ be the trivial character of $T_w^F$. The Green function $Q^H_{T_w}(u)$ is defined as the restriction of $R^H_{T_w}\mathbf{1}$ to $H^F_u$, the subset of unipotent elements. These functions only depend on the $F$-conjugacy class of $w$.
\begin{Prop}(\cite[Proposition 12.2]{DM91})
Let $s\in H^F$ be a semi-simple element and $u\in H^F$ a unipotent element that commutes with $s$. Let $\theta$ be a $\ladic$-valued function on $T^F_w$. We have
\begin{equation}\label{eq-char-formula}
R^H_{T_w}\theta(us)=\frac{1}{|C_H(s)^{\circ F}|}\sum_{\{h\in H^{F}\mid hsh^{-1}\in T_{w}\}}Q^{C_H(s)^{\circ}}_{C_{h^{-1}T_{w}h}(s)}(u)\theta(hsh^{-1}).
\end{equation}
\end{Prop}

\subsubsection{}\label{A^F_tau}
Fix a semi-simple element $s\in H^F$. For any $\tau\in W_H$, put $$A_{\tau}=\{h\in H\mid hsh^{-1}\in T_{\tau}\},\quad A^F_{\tau}=A_{\tau}\cap H^F.$$ The set $A^F_{\tau}$ is crucial for the computation of Deligne-Lusztig characters. If $A^F_{\tau}$ is non empty, we may assume that $s\in T_{\tau}^F$, since $R^H_{T_w}\theta$ is invariant under conjugation by $H^F$. Put $L:=C_H(s)^{\circ}$. For any $h\in A^F_{\tau}$, $C_{h^{-1}T_{\tau}h}(s)$ is an $F$-stable maximal torus of $L$. In particular, $T_{\tau}$ is an $F$-stable maximal torus of $L$. Let $W_L:=W_L(T_{\tau})$ be the Weyl group of $L$ defined by $T_{\tau}$. The $L^F$-conjugacy classes of the $F$-stable maximal tori of $L$ are parametrised by the $F$-conjugacy classes of $W_L$. There is a natural map from $A_{\tau}^F$ to the set of $F$-conjugacy classes of $W_L$, sending $h\in A^F_{\tau}$ to the class of $C_{h^{-1}T_{\tau}h}(s)$. Denote by $B_{\tau}$ the image of this map. Let $\nu\in W_L$ represente an $F$-conjugacy class of $W_L$ and denote by $A^F_{\tau,\nu}$ the inverse image in $A^F_{\tau}$ of this $F$-conjugacy class.

Let $g_{\tau}\in H$ such that $T_{\tau}=g_{\tau}Tg_{\tau}^{-1}$. Write $s_0=g_{\tau}^{-1}sg_{\tau}$ and $L_0=C_H(s_0)$. Let $W_{L_0}$ be the Weyl group of $L_0$ defined by $T$. The $F$-conjugacy classes of $W_L$ are in natural bijection with the $\tau$-conjugacy classes of $W_{L_0}$. Two elements $v$ and $w$ are $\tau$-conjugate if there is some $x\in W_L$ such that $v=xw\tau x^{-1}\tau^{-1}$. Therefore, we can identify $B_{\tau}$ with a set of $\tau$-conjugacy classes of $W_{L_0}$.
\begin{Prop}\label{card-Ataunu}
We have $$|A_{\tau,\nu}^F|=|L^F|z_{\tau}z^{-1}_{\nu},$$where $z_{\tau}$ is the cardinality of the centraliser of $\tau$ in $W_H$, and $z_{\nu}$ is the cardinality of the stabiliser of $\nu$ under the $\tau$-twisted conjugation by $W_{L_0}$.
\end{Prop}
\begin{proof}
See \cite[Equation (4.3.3)]{HLR}. Note that the $\bar{A}_{\nu}$ therein is a subset of the quotient $A^F_{\tau}/L^F$.
\end{proof}

\subsection{$\sigma$-stable irreducible characters}\hfill

Write $G=\GL_n$. Denote by $T\subset G$ the maximal torus consisting of diagonal matrices, $B\subset G$ the Borel subgroup consisting of upper triangular matrices. Denote by $\Delta$ the simple roots of $G$ defined by $(T,B)$.

\subsubsection{}
Let $M\subset G$ be an $F$-stable Levi subgroup and let $T_M\subset M$ be an $F$-stable maximal torus. Let $W_M$ be the Weyl group of $M$ defined by $T_M$. For any $w\in W_M$, we choose an $F$-stable maximal torus $T_w\subset M$ whose $M^F$-conjugacy class corresponds to the $F$-conjugacy class of $w$. Let $\theta$ be a regular linear character of $M^F$ (See \cite[\S 3.1 (a), (b)]{LS}). Since $F$ preserves $T_M$, it acts on $W_M$. Let $\varphi\in\Irr(W_M)^F$. It extends to a character $\tilde{\varphi}\in\Irr(W_M\rtimes\lb F\rb)$. Such an extension is not unique. Write 
\begin{equation}\label{eq-alm-char}
R^G_{\varphi}\theta:=|W_M|^{-1}\sum_{w\in W_M}\tilde{\varphi}(wF)R^G_{T_w}\theta_{T_w},
\end{equation}
 where $\theta_{T_w}$ is the restriction of $\theta$ to $T_w^F$. By \cite[Theorem 3.2]{LS}, for some choice of $\tilde{\varphi}$, the function $\epsilon_G\epsilon_MR^G_{\varphi}\theta$ is an irreducible character of $\GL_n(q)$.

\subsubsection{}\label{GLnSigma-F-Levi}
For any subset $I\subset\Delta$, denote by $L_I$ the corresponding standard Levi subgroup of $G$ with respect to $(T,B)$. If $L_I$ is $\sigma$-stable, then we may choose an isomorphism $$L_I\cong\GL_{n_0}\times\prod_i(\GL_{n_i}\times\GL_{n_i}).$$Write $W_I:=N_G(L_I)/L_I$. Since $\sigma$ leaves $L_I$ stable, it induces an automorphism of $W_I$. As is explained in \cite[\S 3.2.2]{Shu3}, there is an isomorphism $$W_{(G^{\sigma})^{\circ}}((L_I^{\sigma})^{\circ})\lisom W_I^{\sigma}.$$ Therefore, for any $w\in W_I^{\sigma}$, we may choose $\dot{w}\in (G^{\sigma})^{\circ}$ representing $w$. If $g\in(G^{\sigma})^{\circ}$ is such that $g^{-1}F(g)=\dot{w}$, then define $L_{I,w}:=gL_Ig^{-1}$. It is an $F$-stable and $\sigma$-stable Levi subgroup of $G$. If $M$ is an $F$-stable and $\sigma$-stable Levi factor of a $\sigma$-stable parabolic subgroup, then $M$ is necessarily $G^F$-conjugate to an $L_{I,w}$ for some $I$ and $w$.

\subsubsection{}\label{sss-sigst}
Let $M=L_{I,w}$ for some $I$ and $w$ as above. Write $M$ as $M_0\times M_1$ with $M_0\cong\GL_{n_0}$ and $M_1\cong\prod_i(\GL_{n_i}\times\GL_{n_i})$, where $n_0$ and the $n_i$'s are as in the previous paragraph. Let $T_M\subset M$ be an $F$-stable and $\sigma$-stable maximal torus. We can write $T_M=T_0\times T_1$ with $T_0\subset M_0$ and $T_1\subset M_1$. We may assume that $T_0\subset M_0$ is a split maximal torus. Denote by $W_0$ (resp. $W_1$) the Weyl group of $M_0$ (resp. $M_1$) defined by $T_0$ (resp. $T_1$). 

Let $n_+$ and $n_-$ be some non negative integers such that $n_++n_-=n_0$. Let $M_{00}\subset M_0$ be a Levi subgroup containing $T_0$ such that $M_{00}^F\cong\GL_{n_+}(q)\times\GL_{n_-}(q)$. Its Weyl group is isomorphic to $\mathfrak{S}_{n_+}\times\mathfrak{S}_{n_-}$.

Let $\theta_1$ be a linear character of $M_1^F$ that is $\sigma$-stable. Then $\theta:=\mathbf{1}\boxtimes\eta\boxtimes\theta_1$ is a linear character of $M_{00}\times M_1$, where $\mathbf{1}$ is the trivial character of $\GL_{n_+}(q)$ and $\eta$ is the order 2 irreducible character of $\mathbb{F}_q^{\ast}$ and is regarded as a character of $\GL_{n_-}(q)$ via the determinant map. The actions of $F$ and $\sigma$ preserve $M_0$ and $M_1$, thus they induce actions on $W_0$ and $W_1$. Let $\varphi_1\in\Irr(W_1)^F\cap\Irr(W_1)^{\sigma}$ and let $(\varphi_+,\varphi_-)\in\Irr(\mathfrak{S}_{n_+}\times\mathfrak{S}_{n_-})$. Then $\varphi:=(\varphi_+,\varphi_-,\varphi_1)$ is an irreducible character of the Weyl group of $M_{00}\times M_1$.

For any connected reductive algebraic group $H$ defined over $\mathbb{F}_q$, define $\epsilon_H:=(-1)^{rk_H}$ where $rk_H$ is the $\mathbb{F}_q$-rank of $H$. Note that $M_{00}\times M_1$ and $M$ have the same $\mathbb{F}_q$-rank.
\begin{Thm}(\cite[Proposition 5.2.2, Proposition 5.2.3]{Shu2})\label{sigstShu2}
Suppose that $\theta$ is a regular linear character of $M_{00}^F\times M_1^F$. Then for some choice of the extension $\tilde{\varphi}$, the function $\epsilon_G\epsilon_MR^G_{\varphi}\theta$ is a $\sigma$-stable irreducible character of $\GL_n(q)$. Moreover, every $\sigma$-stable irreducible character of $\GL_n(q)$ is of this form.
\end{Thm}

\subsubsection{}\label{sss-reg-lin}
Let $M=L_{I,w}$ with $w=1$, i.e. a standard Levi subgroup. Then $\epsilon_G=\epsilon_M$ and $M_1^F\cong\prod_i(\GL_{n_i}(q)\times\GL_{n_i}(q))$ and $F$ acts trivially on $W_1\cong\prod_i(\mathfrak{S}_{n_i}\times\mathfrak{S}_{n_i})$. A $\sigma$-stable linear character of $M_1^F$ is of the form $(\theta_i,\theta_i^{-1})_i$, where $\theta_i$ is a linear character of $\GL_{n_i}(q)$ for each $i$. The linear character $\theta$ in \S \ref{sss-sigst} is regular if and only if $\theta_i\ne\theta_j^{\pm}$ whenever $i\ne j$ and $\theta_i^2\ne \mathbf{1}$ for all $i$. The set of these $\sigma$-stable linear characters of $M_1^F$ is denoted by $\Irr^{\sigma}_{\reg}(M_1^F)$.

\subsection{Types}\hfill

\textit{Types} are some combinatorial data that are used to describe $\sigma$-stable irreducible characters associated to standard Levi subgroups.

\subsubsection{}\label{sss-types}
Denote by $\mathfrak{T}$ the set of data of the form $$\bs\omega=\omega_+\omega_-(\omega_i)_i,$$where $\omega_+$ and $\omega_-$ are partitions and $(\omega_i)_i$ is an unordered sequence of non trivial partitions. An element of $\mathfrak{T}$ is called a type. For any type $\bs\omega$, we denote by $\bs\omega_{\ast}=(\omega_i)_i$ the part without the components $\omega_+$ and $\omega_-$, so we may write $\bs\omega=\omega_+\omega_-\bs\omega_{\ast}$. The sequence $\bs\omega_{\ast}$ is equivalent to a sequence $(m_{\lambda})_{\lambda\in\mathcal{P}}$, where $\mathcal{P}$ is the set of partitions and $m_{\lambda}$ is the multiplicity of $\lambda$ in $\bs\omega_{\ast}$. The size of a type is defined by $|\bs\omega|:=|\omega_+|+|\omega_-|+\sum_i|\omega_i|$. For any $a\in\mathbb{Z}_{\ge0}$, we denote by $\mathfrak{T}(a)$ the subset of types of size $a$. If we require $\bs\omega_{\ast}$ to be an ordered sequence of partitions, then we call $\bs\omega$ an \textit{ordered type}. We denote by $\tilde{\mathfrak{T}}$ the set of ordered types. The subsets of $\tilde{\mathfrak{T}}$ and $\mathfrak{T}$ consisting of elements with $\omega_+=\omega_-=\varnothing$ are denoted by $\tilde{\mathfrak{T}}_{\ast}$ and $\mathfrak{T}_{\ast}$. 

Denote by $l(\bs\omega_{\ast})$ the length of $\bs\omega_{\ast}$. Write $\bs\omega_{\ast}=(m_{\lambda})_{\lambda\in\mathcal{P}}$ and define
\begin{equation}
N(\bs\omega_{\ast})=\prod_{\lambda}m_{\lambda}!,
\end{equation}
and
\begin{equation}
K(\bs\omega_{\ast})=(-1)^{l(\bs\omega_{\ast})}l(\bs\omega_{\ast})!.
\end{equation}
Define $\{\bs\omega_{\ast}\}:=(2m_{\lambda})_{\lambda\in\mathcal{P}}$ and $\{\bs\omega\}:=\omega_+\omega_-\{\bs\omega_{\ast}\}$. Define a map $[]:\mathfrak{T}\rightarrow\mathcal{P}$ by requiring that the parts of $[\bs\omega]$ are the union of the parts of the partitions $\omega_+$, $\omega_-$ and $\omega_i$'s.

\subsubsection{}
Denote by $\tilde{\mathfrak{T}}_s$ the subset of $\tilde{\mathfrak{T}}$ consisting of elements of the form $$\bs\omega=(1^{n_+})(1^{n_-})((1^{n_i}))_i.$$ We define its dual type by $\bs\omega^{\ast}=(n_+)(n_-)((n_i))_i$. Elements of $\tilde{\mathfrak{T}}_s$ can also be represented by a sequence of integers: $\bs\omega=n_+n_-(n_i)_i$ if $\bs\omega=(1^{n_+})(1^{n_-})((1^{n_i}))_i$. There is a surjective map from $\tilde{\mathfrak{T}}$ to $\tilde{\mathfrak{T}}_s$ sending $\bs\omega=\omega_+\omega_-(\omega_i)_i$ to $|\omega_+||\omega_-|(|\omega_i|)_i$. For $\bs\omega\in\tilde{\mathfrak{T}}_s$, we denote by $\tilde{\mathfrak{T}}(\bs\omega)$ the inverse image of $\bs\omega$ under this map. For $\bs\alpha$, $\bs\beta\in\tilde{\mathfrak{T}}$, we write $\bs\alpha\thickapprox\bs\beta$, if they lie in the same $\tilde{\mathfrak{T}}(\bs\omega)$.

Let $s\in\GL_n$ be a semi-simple elements. Suppose that the multiplicities of its eigenvalues are given by a sequence of integers $(n_i)_i$. Then we say the conjugacy class of $s$ is of type $\bs\omega=\bs\omega_{\ast}=((1^{n_i}))_i$.

\subsubsection{}
For any $\chi\in\Irr(\GL_n(q))^{\sigma}$, there exists some $F$-stable and $\sigma$-stable Levi subgroup $M$ such that $\chi=\epsilon_G\epsilon_MR^G_{\varphi}\theta$ as in Theorem \ref{sigstShu2}. This character is induced from a smaller Levi subgroup $M_{00}\times M_1$. Denote by $\Irr^{\sigma}_{st}\subset\Irr(\GL_n(q))^{\sigma}$ the subset of characters such that $M$ can be chosen to be a standard Levi subgroup. Define a map 
\begin{equation}
\pi:\Irr^{\sigma}_{st}\longrightarrow\mathfrak{T}
\end{equation}
as follows. Let $\varphi=(\varphi_+,\varphi_-,\varphi_1)$ as in \S \ref{sss-sigst}, and let $\omega_+$ and $\omega_-$ be the partitions corresponding to the characters $\varphi_+$ and $\varphi_-$. Since the Weyl group of $M_1$ is isomorphic to $\prod_i(\mathfrak{S}_{n_i}\times\mathfrak{S}_{n_i})$, $\varphi_1$ gives a sequence of partitions $\bs\omega_{\ast}$. Each partition in $\bs\omega_{\ast}$ has even multiplicity. Let $\bs\omega'_{\ast}$ be such that $\{\bs\omega'_{\ast}\}=\bs\omega_{\ast}$. Then we define $\pi(\chi):=\omega_+\omega_-\bs\omega'_{\ast}$. Define 
\begin{equation}
\Irr^{\sigma}_{\bs\omega}:=\pi^{-1}(\bs\omega).
\end{equation}
It is empty unless $|\{\bs\omega\}|=n$. Its elements are called $\sigma$-stable irreducible characters of type $\bs\omega$. 

For all $\sigma$-stable characters of the same type, the Levi subgroup $M$ in Theorem \ref{sigstShu2} can be chosen to be the same, and the characters $\varphi$ are the same by definition. Therefore, there is a surjective map
\begin{equation}\label{eq-reg->omega}
\Irr^{\sigma}_{\reg}(M_1^F)\longrightarrow \Irr^{\sigma}_{\bs\omega}.
\end{equation}
The cardinality of the fibre is equal to $2^{l(\bs\omega_{\ast})}N(\bs\omega_{\ast})$ as is explained in \cite[\S 3.5.2]{Shu3}.

\subsubsection{}
Let $L$ be a Levi subgroup of $\GL_n$ containing a split maximal torus $T$. Denote by $W_L$ the Weyl group of $L$ defined by $T$. By choosing an isomorphism $L\cong\GL_{n_+}\times\GL_{n_-}\times\prod_i\GL_{n_i}$, we can parametrise the conjugacy classes of $W_L$ by $\tilde{\mathfrak{T}}(\bs\omega)$, with $\bs\omega=(1^{n_+})(1^{n_-})((1^{n_i}))_i$. (If $L$ is the centraliser of a semi-simple element, then we make the convention that $n_+=n_-=0$.) Then the map $[]:\mathfrak{T}\rightarrow\mathcal{P}$ is simply the map between conjugacy classes induced by the inclusion $W_L\hookrightarrow W$ where $W\cong\mathfrak{S}_n$ is the Weyl group of $\GL_n$ defined by $T$.

Suppose that $\bs\omega$ is the type of an element $s\in T^F$. Let $B_{\tau}$ be the set defined in \S \ref{A^F_tau}.
\begin{Lem}
Let $\bs\tau$ be the partition corresponding to the conjugacy class of $\tau$. Then $$B_{\tau}=\{\bs\nu\in\tilde{\mathfrak{T}}(\bs\omega)\mid [\bs\nu]=\bs\tau\}.$$
\end{Lem}
\begin{proof}
Unwind the definitions.
\end{proof}

%%%%%%%%%%%%%%%%%%%%%%
\section{$\GL_n\lb\sigma\rb~$-Character Varieties}\label{Section-Char-Var}
In this section, we give the definition of $\GL_n\lb\sigma\rb~$-character varieties and the definition of generic conjugacy classes. We also explain how to pass to finite fields.

\subsection{Definition of $\GL_n\lb\sigma\rb~$-Character Varieties}\label{Defn-Char.Var}

\subsubsection{}
Let $p':\tilde{\Sigma}'\rightarrow \Sigma'$ be an unbranched double covering of compact Riemann surfaces, and let $\mathcal{R}=\{x_j\}_{1\le j\le k}\subset\Sigma'$ be a finite set. Put $\Sigma=\Sigma'\setminus\mathcal{R}$ and $\tilde{\Sigma}=p^{\prime -1}(\Sigma)$ and denote by $p:\tilde{\Sigma}\rightarrow \Sigma$ the restriction of $p'$. Fix the base points of $\tilde{\Sigma}$ and $\Sigma$ so that we have an injective homomorphism of fundamental groups $\pi_1(\tilde{\Sigma})\rightarrow\pi_1(\Sigma)$. Let $\Rep(\Sigma)$ be the space of homomorphisms $\rho$ that make the following diagram commute:
$$
\begin{tikzcd}[row sep=2.5em, column sep=0.2em]
\pi_1(\Sigma) \arrow[dr, swap, "q_1"] \arrow[rr, "\rho"] && \GL_n\lb\sigma\rb \arrow[dl, "q_2"]\\
& \Gal(\tilde{\Sigma}/\Sigma)\cong\bs\mu_2
\end{tikzcd}
$$ 
where $q_1$ is the quotient by $\pi_1(\tilde{\Sigma})$ and $q_2$ is the quotient by $\GL_n$. Our $\GL_n\lb\sigma\rb~$-character variety $\Ch(\Sigma)$ is defined as the categorical quotient of $\Rep(\Sigma)$ for the conjugation action of $\GL_n$ on $\GL_n\lb\sigma\rb$.

\subsubsection{}
Denote by $g$ the genus of $\Sigma'$. We may choose the generators $\alpha_i$, $\beta_i$, $1\le i\le g$, and $\gamma_j$, $1\le j\le k$, of $\pi_1(\Sigma)$ that satisfy the relation 
\begin{equation}
\prod_{i=1}^g[\alpha_i,\beta_i]\prod_{j=1}^k\gamma_j=1,
\end{equation}
and the condition: the generators $\beta_1$, $\alpha_i$, $\beta_i$, $2\le i\le g$, and $\gamma_j$, $1\le j\le k$, are the images of some elements of $\pi_1(\tilde{\Sigma})$, while $\alpha_1$ lie in $\pi_1(\Sigma)\setminus\pi_1(\tilde{\Sigma})$. Let $\mathcal{C}=(C_j)_{1\le j\le k}$ be a $k$-tuple of semi-simple conjugacy classes contained in $\GL_n$. We define the subvariety $\Rep_{\mathcal{C}}(\Sigma)\subset\Rep(\Sigma)$ as consisting of $\rho\in\Rep(\Sigma)$ such that $\rho(\gamma_j)\in C_j$ for all $j$. Then we define $\Ch_{\mathcal{C}}(\Sigma):=\Rep_{\mathcal{C}}(\Sigma)\ds\GL_n$. The defining equation of $\Rep_{\mathcal{C}}(\Sigma)$ is then
\begin{equation}\label{CharVarEq}
\Rep_{\mathcal{C}}(\Sigma)=\left\{(A_i,B_i)_i(X_j)_j\in \GL_n^{2g}\times \prod_{j=1}^{k} C_j\mid A_1\sigma(B_1)A_1^{-1}B_1^{-1}\prod_{i=2}^g[A_i,B_i]\prod_{j=1}^{k}X_j=1\right\}.
\end{equation}

\subsection{Generic Conjugacy Classes}\label{ss-GCC}
\subsubsection{}\label{GCC}
Let $\mathcal{C}=(C_j)_{1\le j\le k}$ be a tuple of semi-simple conjugacy classes in $\GL_n$. For each $j$, let $t_j\in T$ be a representative of $C_j$, and write $t_j=(a_{j,1},\ldots,a_{j,n})$ with each $a_{j,\gamma}\in\mathbb{C}^{\ast}$. Write $\Lambda=\{1,\ldots,n\}$. For any $j$ and any subset $\mathbf{A}\subset\Lambda$, write $\smash{[\mathbf{A}]_j=\prod_{\gamma\in\mathbf{A}}a_{j,\gamma}}$. Following \cite[\S 4.4]{Shu1}, we say that $\mathcal{C}$ is generic if for any $1\le M \le [n/2]$, any $2k$-tuple $(\mathbf{A}_1,\ldots,\mathbf{A}_{k},\mathbf{B}_1,\ldots,\mathbf{B}_{k})$ of subsets of $\Lambda$ such that 
\begin{itemize}
\item
$|\mathbf{A}_1|=\cdots=|\mathbf{A}_{k}|=|\mathbf{B}_1|=\cdots=|\mathbf{B}_{k}|=M$;
\item
$\mathbf{A}_j\cap \mathbf{B}_j=\varnothing$, for all $j$,
\end{itemize}
 we have
\begin{equation}\label{GenConGLn1}
[\mathbf{A}_1]_1\cdots[\mathbf{A}_{k}]_{k}[\mathbf{B}_1]_1^{-1}\cdots[\mathbf{B}_{k}]_{k}^{-1}\ne1.
\end{equation}

We say that $\mathcal{C}$ is \textit{strongly generic} if for any $M$ and $(\mathbf{A}_1,\ldots,\mathbf{A}_{k},\mathbf{B}_1,\ldots,\mathbf{B}_{k})$ as above, we have
\begin{equation}\label{GenConGLn2}
[\mathbf{A}_1]_1\cdots[\mathbf{A}_{k}]_{k}[\mathbf{B}_1]_1^{-1}\cdots[\mathbf{B}_{k}]_{k}^{-1}\ne\pm1.
\end{equation}
\begin{Eg}
In \cite{HLR}, one may take a central conjugacy class that is also generic. But in our situation, this is not possible. See \cite[Remark 3]{Shu1}.
\end{Eg}

In fact, according to the above definition, in order for the tuple $\mathcal{C}$ to be generic, it is necessary that there exists a regular conjugacy class among the $C_j$'s.

\subsubsection{}\label{explicit-genconj-Fq}
We can also define generic conjugacy classes in the finite group $\GL_n(q)\lb\sigma\rb$. Let $T$ be the maximal torus consisting of diagonal matrices. Let $\mathcal{C}=(C_j)_{1\le j\le k}$ be a tuple of semi-simple conjugacy classes in $\GL_n(q)$, such that each $C_j$ has a representative $t_j\in T^F$. Note that not all semi-simple conjugacy classes of $\GL_n(q)$ have representatives of this form. For each $j$, let $(a_{j,1},\ldots,a_{j,n})\in(\mathbb{F}_q^{\ast})^n$ be the eigenvalues of $t_j$. With the same notations as in the previous paragraph, we say that $\mathcal{C}$ is generic (resp. strongly generic) if for any $1\le M \le [n/2]$, $(\mathbf{A}_1,\ldots,\mathbf{A}_{k},\mathbf{B}_1,\ldots,\mathbf{B}_{k})$, we have
\begin{equation}
[\mathbf{A}_1]_1\cdots[\mathbf{A}_{k}]_{k}[\mathbf{B}_1]_1^{-1}\cdots[\mathbf{B}_{k}]_{k}^{-1}\ne1\text{ (resp. $\pm1$)}.
\end{equation}

\begin{Lem}\label{disct-e.v.-genconj}
Let $(C_j)_{1\le j\le k}$ be a generic tuple of conjugacy classes, with eigenvalues given by $(a_{j,1},\ldots,a_{j,n})$. Let $M$ and $M'$ be positive integers such that $M+M'\le[n/2]$. Let $$(\mathbf{A}_1,\ldots,\mathbf{A}_{k},\mathbf{B}_1,\ldots,\mathbf{B}_{k})~\text{and }(\mathbf{A}'_1,\ldots,\mathbf{A}'_{k},\mathbf{B}'_1,\ldots,\mathbf{B}'_{k})$$ be two $2k$-tuples of subsets of $\Lambda$ such that
\begin{itemize}
\item
$|\mathbf{A}_1|=\cdots=|\mathbf{A}_{k}|=|\mathbf{B}_1|=\cdots=|\mathbf{B}_{k}|=M$;
\item
$|\mathbf{A}'_1|=\cdots=|\mathbf{A}'_{k}|=|\mathbf{B}'_1|=\cdots=|\mathbf{B}'_{k}|=M'$;
\item
The sets $\mathbf{A}_j$, $\mathbf{B}_j$, $\mathbf{A}'_j$ and $\mathbf{B}'_j$ are mutually disjoint, for all $j$.
\end{itemize}
Then,$$[\mathbf{A}_1]_1\cdots[\mathbf{A}_{k}]_{k}[\mathbf{B}_1]_1^{-1}\cdots[\mathbf{B}_{k}]_{k}^{-1}\ne\left([\mathbf{A}'_1]_1\cdots[\mathbf{A}'_{k}]_{k}[\mathbf{B}'_1]_1^{-1}\cdots[\mathbf{B}'_{k}]_{k}^{-1}\right)^{\pm 1}.$$
\end{Lem}
\begin{proof}
Suppose $$[\mathbf{A}_1]_1\cdots[\mathbf{A}_{k}]_{k}[\mathbf{B}_1]_1^{-1}\cdots[\mathbf{B}_{k}]_{k}^{-1}\ne\left([\mathbf{A}'_1]_1\cdots[\mathbf{A}'_{k}]_{k}[\mathbf{B}'_1]_1^{-1}\cdots[\mathbf{B}'_{k}]_{k}^{-1}\right)^{\epsilon},$$with $\epsilon=1$. For each $j$, put $\tilde{\mathbf{A}}_j=\mathbf{A}_j\sqcup\mathbf{B}'_j$ and $\tilde{\mathbf{B}}_j=\mathbf{B}_j\sqcup\mathbf{A}'_j$. Then,$$[\tilde{\mathbf{A}}_1]_1\cdots[\tilde{\mathbf{A}}_{k}]_{k}[\tilde{\mathbf{B}}_1]_1^{-1}\cdots[\tilde{\mathbf{B}}_{k}]_{k}^{-1}=1,$$which is a contradiction. The case $\epsilon=-1$ is similar.
\end{proof}

\subsection{The $R$-Model}\hfill

As in \cite{Shu3}, we define the character variety over a particular base ring so that when passing to finite fields the conjugacy classes remain generic.
\subsubsection{}\label{R}
The conjugacy classes $(C_j)_j$ are represented by some $n$-tuples of complex numbers $$(a^j_1,\ldots,a^j_1,\ldots,a^j_{l_j},\ldots,a^j_{l_j}),$$satisfying
\begin{itemize}
\item[(i)] $a^j_r\ne a^j_s$ for any $j$ and any $r\ne s$;
\item[(ii)] $[\mathbf{A}_1]_1\cdots[\mathbf{A}_{k}]_{k}[\mathbf{B}_1]_1^{-1}\cdots[\mathbf{B}_{k}]_{k}^{-1}\ne1$,\\  for any $1\le M \le [n/2]$, any $2k$-tuple $(\mathbf{A}_1,\ldots,\mathbf{A}_{k},\mathbf{B}_1,\ldots,\mathbf{B}_{k})$ of subsets of $\Lambda$ such that 
\begin{itemize}
\item
$|\mathbf{A}_1|=\cdots=|\mathbf{A}_{k}|=|\mathbf{B}_1|=\cdots=|\mathbf{B}_{k}|=M$;
\item
$\mathbf{A}_j\cap \mathbf{B}_j=\varnothing$, for all $j$.
\end{itemize}
\end{itemize} 
For each $j$ and each $1\le r\le l_j$, choose $b^j_r\in\mathbb{C}$ such that $(b^j_r)^2=a^j_r$. Denote by $R_0$ the subring of $\mathbb{C}$ generated by $\{(b^j_r)^{\pm 1}\mid\text{all }r,j\}$. Let $S\subset R_0$ be the multiplicative subset generated by
\begin{itemize}
\item[(i)] $a^j_{r}-a^j_{s}$ for any $j$ and $r\ne s$;
\item[(ii)] $[\mathbf{A}_1]_1\cdots[\mathbf{A}_{k}]_{k}[\mathbf{B}_1]_1^{-1}\cdots[\mathbf{B}_{k}]_{k}^{-1}-1$, for any $M$ and $(\mathbf{A}_1,\ldots,\mathbf{A}_{k},\mathbf{B}_1,\ldots,\mathbf{B}_{k})$ as above.
\end{itemize}
Define the ring of generic eigenvalues as $R:=S^{-1}R_0$. We will see below that the character variety is defined over $R$. 

Denote by $\mu_j=(m_{j,r})_{1\le r\le l_j}\in\mathfrak{T}_s$ the type of $C_j$, so that $m_{j,r}$ is the multiplicity of $a^j_r$.

\subsubsection{}\label{A_0}
Let $\mathcal{A}_0$ be the polynomial ring over $R$ with $n^2(2g+k)$ indeterminates thought of as the entries of some $n\times n$ matrices $A_1,B_1,\ldots A_g,B_g,X_1,\ldots X_{k}$, with $\det A_i$, $\det B_i$, $\det X_j$, $1\le i\le g$, $1\le j\le k$, inverted. Let $I_0\subset \mathcal{A}_0$ be the ideal generated by 
\begin{itemize}
\item[(i)] The entries of $A_1\sigma(B_1)A_1^{-1}B_1^{-1}[A_2,B_2]\cdots[A_g,B_g]X_1\cdots X_{k}-\Id$ (Note that $\sigma$ is defined over $\mathbb{Z}$);
\item[(ii)] For all $1\le j\le k$, the entries of $\prod_{r=1}^{l_j}(X_j-a^j_r\Id)$;
\item[(iii)] For all $1\le j\le k$, the entries of the coefficients of the following polynomial in an auxiliary variable $t$:
\begin{equation}
\det(t\Id-X_j)-\prod_r^{l_j}(t-a^j_r)^{m_{j,r}}.
\end{equation}
\end{itemize} 

Define $\mathcal{A}:=\mathcal{A}_0/\sqrt{I_0}$. Then $\sRep_{\mathcal{C}}:=\Spec\mathcal{A}$ is the $R$-model of $\Rep_{\mathcal{C}}$. Let $\GL_n$ act on $B_1$, $A_i$, $B_i$, $2\le i\le g$, and $X_j$, $1\le j\le k$ by conjugation, and act on $A_1$, by the $\sigma$-twisted conjugation. Then $\sCh_{\mathcal{C}}:=\Spec\mathcal{A}^{\GL_n(R)}$ is the $R$-model of $\Ch_{\mathcal{C}}$, since taking invariants commutes with flat base change (\cite[\S I.2 Lemma 2]{Ses77}).

\subsubsection{}
Let $\phi:R\rightarrow \mathbb{F}_q$ be any ring homomorphism. For each $1\le j\le k$, denote by $C_j^{\phi}$ the subvariety of $\GL_n$ over $\mathbb{F}_q$ defined by \S \ref{A_0} (ii), (iii). Denote by $\Rep^{\phi}_{\mathcal{C}}$ (resp. $\Ch_{\mathcal{C}}^{\phi}$) the variety over $\mathbb{F}_q$ obtained by base change from $\sRep_{\mathcal{C}}$ (resp. $\sCh_{\mathcal{C}}$). We have
\begin{equation}\label{RepCFq}
\Rep^{\phi}_{\mathcal{C}}(\mathbb{F}_q)=\{(A_i,B_i)(X_j)\in \GL_n(q)^{2g}\times \prod^{k}_{j=1} C_j^{\phi}(\mathbb{F}_q)\mid A_1\sigma(B_1)A_1^{-1}B_1^{-1}\prod_{i=2}^g[A_i,B_i]\prod_{j=1}^{k}X_j=1\},
\end{equation}
where $C_j^{\phi}(\mathbb{F}_q)$ is a conjugacy class in $\GL_n(q)$ with eigenvalues $\phi(a^j_r)$, $1\le r\le l_j$. By the definition of the base ring $R$, the tuple of conjugacy classes $\mathcal{C}^{\phi}$ is also generic.
\begin{Rem}\label{ajr-in-2}
Our construction of $R$ guarantees that $\phi(a^j_r)\in(\mathbb{F}_q^{\ast})^2$.
\end{Rem}

\begin{Prop}\label{M=U/G}
We have,
\begin{equation}
|\Ch^{\phi}_{\mathcal{C}}(\mathbb{F}_q)|=\frac{1}{|\GL_n(\mathbb{F}_q)|}|\Rep_{\mathcal{C}}^{\phi}(\mathbb{F}_q)|.
\end{equation}
\end{Prop}
\begin{proof}
The same as \cite[Proposition 4.3.3]{Shu3}.
\end{proof}

%%%%%%%%%%%%%%%%%%%%%%
\section{Two Lemmas}\label{Sec-Lin}
In this section, we denote by $T$ the maximal torus consisting of diagonal matrices and denote by $W$ the Weyl group of $\GL_n$ defined by $T$. 
\subsection{A Constraint on Types}
\subsubsection{}\label{sss-1st-lemma}
We show that only a particular subset of the set of $\sigma$-stable irreducible characters of $\GL_n(q)$ contributes to the E-polynomial.
\begin{Lem}\label{restrictTypes}
Let $(C_1,\ldots,C_{k})$ be a tuple of semi-simple conjugacy classes in $\GL_n(q)$ such that each $C_j$ has a representative $s_j\in T^{F}$. Let $M$ be a $\sigma$-stable and $F$-stable Levi factor of a $\sigma$-stable parabolic subgroup of $\GL_n(\kk)$. Suppose that $(C_j)_j$ is strongly generic and that for all $j$, we have $C_j\cap M\ne\varnothing$. Then $M$ is $\GL_n(q)$-conjugate to a $\sigma$-stable standard Levi subgroup.
\end{Lem}
\begin{proof}
We may assume that $M=L_{I,w}$ for some $\sigma$-stable $I$ and some $w\in W_I^{\sigma}$ (see \S \ref{GLnSigma-F-Levi}). There exists $\dot{w}\in(G^{\sigma})^{\circ}$ representing $w$ and $g\in(G^{\sigma})^{\circ}$ such that $g^{-1}F(g)=\dot{w}$. Then $L_{I,w}=gL_Ig^{-1}$.

For each $j$, let $t_j\in C_j\cap M$. For any $j$, $g^{-1}t_jg$ lies in $L_I$. For each $j$, there exists $l_j\in L_I$ such that $l_jg^{-1}t_jgl_j^{-1}$ lies in $T$, and moreover, for each $j$, there exists $w_j\in W$ such that $$l_jg^{-1}t_jgl_j^{-1}=w_js_jw_j^{-1}.$$For any connected reductive algebraic group $H$, denote by 
$$\mathbf{D}_H:H\longrightarrow H/[H,H]\cong Z_H^{\circ}/(Z_H^{\circ}\cap[H,H])$$
the natural surjection. We have  $$\mathbf{D}_{L_I}\left((l_jg^{-1}t_jgl_j^{-1})\sigma(l_jg^{-1}t_jgl_j^{-1})\right)=\mathbf{D}_{L_I}(g^{-1}t_j\sigma(t_j) g).$$ We deduce that 
\begin{equation}\label{l}
\mathbf{D}_{L_I}\big(\prod_{j=1}^{2k}g^{-1}t_j\sigma(t_j) g\big)=\mathbf{D}_{L_I}\big(\prod_{j=1}^{2k}(w_js_jw_j^{-1})\sigma(w_js_jw_j^{-1})\big),
\end{equation}
and we denote by $l$ this element of $Z_{L_I}^{\circ}/(Z_{L_I}^{\circ}\cap[L_I,L_I])$. 

We may choose an isomorphism $L_I\cong \GL_{n_0}\times\prod_i(\GL_{n_i}\times\GL_{n_i})$, so that $$Z_{L_I}^{\circ}/(Z_{L_I}^{\circ}\cap[L_I,L_I])\cong \kk^{\ast}\times \prod_{i}(\kk^{\ast}\times \kk^{\ast}).$$ The action of $\sigma$ on the factor $\kk^{\ast}\times\kk^{\ast}$ corresponding to each $i$ sends $(x,y)$ to $(y^{-1},x^{-1})$, and on the first factor sends  $x$ to $x^{-1}$. From the right hand side of (\ref{l}), we see that $l$ is a fixed point of $\sigma$. We can then write $l=(l_0,(l_i,l_i^{-1}))$ under the above isomorphism.

In view of the right hand side of (\ref{l}), each $l_i$ is of the form $$[\mathbf{A}_1]_1\cdots[\mathbf{A}_{k}]_{k}[\mathbf{B}_1]_1^{-1}\cdots[\mathbf{B}_{k}]_{k}^{-1}$$ (See \S \ref{GCC} for the notations). We can then apply Lemma \ref{disct-e.v.-genconj} and conclude that $l_i\ne l_j^{\pm 1}$ whenever $i\ne j$. Moreover, all $l_i$ lie in $\mathbb{F}_q$. However, the left hand side of (\ref{l}) shows that $l$ is an $F_w$-stable element (see (\ref{F_w})). According to \cite[\S 3.2.2]{Shu3}, $W_I^{\sigma}$ is of the form $\prod_r\mathfrak{W}_{N'_r}$ for some positive integers $N'_r$, so $w$ consists of some signed cycles. It is easy to see that if there is some signed cycle in $w$ that has size larger than $1$, then $l_i=l_j$ for some $i\ne j$, which is a contradiction. Finally, if there is some negative cycle of size $1$, then the corresponding factor $l_i$ satisfies $l_i=l_i^q=l_i^{-1}$ and so must be equal to $\pm 1$. This is impossible due to the strongly generic condition (\ref{GenConGLn2}). We conclude that $w=1$.
\end{proof}
Let $\chi\in\Irr(\GL_n(q))^{\sigma}$ and let $\mathcal{C}$ be as in the above lemma. Suppose that $\chi\notin\Irr^{\sigma}_{st}$, and that it is obtained from a Levi subgroup $M$ as in \S \ref{sss-sigst}. By the above lemma, there exists some $C_j$ such that $C_j\cap M=\varnothing$. The definition of $R^G_{\varphi}\theta$ and (\ref{eq-char-formula}) imply that $\chi$ must vanish on $C_j$.

\subsection{Sum of linear characters}
\subsubsection{}\label{Irr-reg-sig}
Let $M=L_I$ be a $\sigma$-stable standard Levi subgroup. As in \S \ref{sss-sigst}, we write $M=M_0\times M_1$, where $M_0=\GL_{n_0}$ and $M_1=\prod_{i=1}^l(\GL_{n_i}\times\GL_{n_i})$ for some integers $l$, $n_0$, and $n_i$. Let $n_+$ and $n_-$ be some non negative integers such that $n_++n_-=n_0$, and let $M_{00}\cong\GL_{n_+}\times\GL_{n_-}$ be a Levi subgroup of $M_0$ containing $M_0\cap T$. Denote by $W_{00}$ (resp. $W_1$) the Weyl group of $M_{00}$ (resp. $M_1$) defined by $T$. For each $j\in\{1,\ldots k\}$, let $$w_j\in W_{00}\times W_1\subset W_M(T),$$and choose an $F$-stable maximal torus $T_{w_j}\subset M_{00}\times M_1$ corresponding to the conjugacy class of $w_j$. Recall that $\Irr_{\reg}^{\sigma}(M_1^F)$ is the set of $\sigma$-stable regular linear characters of $M_1^F$ defined in \S \ref{sss-reg-lin}. An element $\theta\in\Irr_{\reg}^{\sigma}(M_1^F)$ determines a linear character $\mathbf{1}\boxtimes\eta\boxtimes\theta$ of $M_{00}^F\times M_1^F$ as in \S \ref{sss-sigst}. For each $j$, we denote by $\theta_j$ the restriction of this linear character to $T_{w_j}^F$. 

Let $\mathcal{C}=(C_1,\ldots,C_{k})$ be a tuple of semi-simple conjugacy classes in $\GL_n(q)$ with representatives $s_j\in T^F$. For each $j$, let $h_j\in\GL_n(q)$ be such that $h_js_jh_j^{-1}\in T_{w_j}$.

\begin{Lem}\label{sumofreg}
Suppose that $\mathcal{C}$ is strongly generic. We have:
\begin{equation}
\sum_{\theta\in\Irr_{\reg}^{\sigma}(M_1^F)}\prod^{k}_{j=1}\theta_j(h_js_j h_j^{-1})=(-2)^ll!.
\end{equation}
\end{Lem}
\begin{proof}
We first give an expression for $\theta_j(h_js_j h_j^{-1})$. Let $g_j\in\GL_n$ be such that $g_jTg_j^{-1}=T_{w_j}$. Write $x=g_j^{-1}h_j$. Then $xs_jx^{-1}\in T$. We deduce that $x=nl$ with $n$ normalising $T$ and $l$ centralising $s_j$. Write $$ns_jn^{-1}=\diag(a_{j,1},\ldots,a_{j,n})\in T^F.$$ Note that this element also lies in $T^{F_{w_j}}$ since $xs_jx^{-1}$ necessarily lies in $T^{F_{w_j}}$. We write $w_j$ as a permutation $\prod_{i\in\Lambda_j}c_{I_{j,i}}$ of $\{1,\ldots,n\}$, where $I_{j,i}$'s are disjoint subsets of $\{1,\ldots,n\}$ indexed by $\Lambda_j$, and $c_{I_{j,i}}$ is a cyclic permutation of $I_{j,i}$. The character $\theta_j\circ\ad g_j$ of $T^{F_{w_j}}$ can then be written as $(\alpha_{j,i})_{i\in\Lambda_j}$, where each $\alpha_{j,i}$ is a character of $\mathbb{F}_q^{\ast}$. With these notations, we have $$\theta(h_js_jh_j^{-1})=\prod_{i\in\Lambda_j}(\alpha_{j,i}(\prod_{\gamma\in I_{j,i}}a_{j,\gamma})).$$(By Remark \ref{ajr-in-2}, the component $\eta$ of $\theta$ has no contribution to this value.)

Let $\theta$ be represented by a tuple $(\theta_1,\ldots,\theta_l)$ of distinct characters of $\mathbb{F}_q^{\ast}$, where $l$ is the number of factors $\GL_{n_i}\times\GL_{n_i}$ in $M_1$. For each $r\in\{1,\ldots,l\}$ and each $j$, put $$I^+_{j,r}:=\bigsqcup_{\{i\in\Lambda_{j}\mid \alpha_{j,i}=\theta_r\}}I_{j,i},\quad I^-_{j,r}:=\bigsqcup_{\{i\in\Lambda_{j}\mid \alpha_{j,i}=\theta_r^{-1}\}}I_{j,i},$$and for each $r$, $$a_r:=\prod_{j=1}^{k}\left(\prod_{\gamma\in I^+_{j,r}}a_{j,\gamma}\prod_{\gamma\in I^-_{j,r}}a^{-1}_{j,\gamma}\right).$$ Then, $$\prod_{j=1}^{k}\prod_{i\in\Lambda_{j}}\alpha_{j,i}(\prod_{\gamma\in I_{j,i}}a_{j,\gamma})=\prod_{r=1}^{l}\theta_r(a_r).$$ 
Therefore, 
$$
\sum_{\theta\in\Irr^{\sigma}_{\reg}(M_1^F)}\prod^{k}_{j=1}\theta_j(h_js_jh_j^{-1})=\sum_{\substack{(\theta_1,\ldots,\theta_l)\\\text{regular}}}\prod^l_{r=1}\theta_r(a_r).
$$ 

Remark \ref{ajr-in-2} says that each $a_r$ is in fact a square in $\mathbb{F}_q^{\ast}$. Therefore we can apply \cite[Lemma 5.2.2, Equation (5.2.1.2), Equation (5.2.2.1)]{Shu3} and obtain$$
\sum_{\substack{(\theta_1,\ldots,\theta_l)\\\text{regular}}}\prod^l_{r=1}\theta_r(a_r)=\sum_{P_1\prec P_0}\mu(P_1,P_0)(-1)^{l(P_1)}2^l=(-2)^ll!.$$
\end{proof}

%%%%%%%%%%%%%%%%%%%%%%%
\section{Computation of E-polynomials}\label{Section-E-poly}  
The goal of this section is to prove Theorem \ref{Main-Thm}.

\subsection{Symmetric functions associated to types}

\subsubsection{}
For any $\bs\omega=\omega_+\omega_-(\omega_i)_i\in\mathfrak{T}$, define the associated Schur symmetric function by $$s_{\bs\omega}(\Z):=s_{\omega_+}(\Z)s_{\omega_-}(\Z)\prod_is_{\omega_i}(\Z).$$Then $p_{\bs\omega}(\Z)$, $m_{\bs\omega}(\Z)$, $h_{\bs\omega}(\Z)$, $P_{\bs\omega}(\Z)$ and $\tilde{H}_{\bs\omega}(\Z)$ can be defined similary. Define $z_{\bs\omega}:=z_{\omega_+}z_{\omega_-}\prod_iz_{\omega_i}$ and $n(\bs\omega):=n(\omega_+)+n(\omega_-)+\sum_i(\omega_i)$. For any $\bs\alpha$, $\bs\beta\in\tilde{\mathfrak{T}}$, if $\bs\alpha\thickapprox\bs\beta$, we define $\chi^{\bs\alpha}_{\bs\beta}:=\chi^{\alpha_+}_{\beta_+}\chi^{\alpha_-}_{\beta_-}\prod_i\chi^{\alpha_i}_{\beta_i}$; otherwise we put $\chi^{\bs\alpha}_{\bs\beta}=0$. If $\bs\alpha\thickapprox\bs\beta$, we define $$K_{\bs\beta,\bs\alpha}(q):=K_{\beta_+,\alpha_+}(q)K_{\beta_-,\alpha_-}(q)\prod_iK_{\beta_i,\alpha_i}(q);$$otherwise, we define $K_{\bs\beta,\bs\alpha}(q)=0$. Similarly we can define $\tilde{K}_{\bs\beta,\bs\alpha}(q)$, so we have $\tilde{K}_{\bs\beta,\bs\alpha}(q)=q^{n(\bs\alpha)}K_{\bs\beta,\bs\alpha}(q^{-1})$.

\begin{Lem}\label{XoXlBCType}
Let $\bs\mu\in\mathfrak{T}_s$. Then the following identity holds:
$$h_{\bs\mu^{\ast}}[\frac{\Z}{1-q}]=(-1)^{|\bs\mu|}q^{-n(\bs\mu)-|\bs\mu|}b_{\bs\mu}(q^{-1})^{-1}\tilde{H}_{\bs\mu}(\Z;q).$$
\end{Lem}
\begin{proof}
See \cite[Lemma 2.3.6]{HLR}.
\end{proof}
\begin{Lem}\label{<s,H>}
Let $\bs\alpha\in\tilde{\mathfrak{T}}$ and $\bs\beta\in\tilde{\mathfrak{T}}_s$. Then
$$
\langle s_{\bs\alpha}(\X),\tilde{H}_{\bs\beta}(\X,q)\rangle=\sum_{\bs\tau\in\tilde{\mathfrak{T}}}\frac{z_{[\bs\tau]}\chi^{\bs\alpha}_{\bs\tau}}{z_{\bs\tau}}\sum_{\substack{\bs\nu\in\tilde{\mathfrak{T}}\\ [\bs\nu]=[\bs\tau] }}\frac{\mathcal{Q}^{\bs\beta}_{\bs\nu}(q)}{z_{\bs\nu}}.
$$
\end{Lem}
\begin{proof}
See \cite[Lemma 2.3.5]{HLR}.
\end{proof}

\subsection{A combinatorial formula for irreducible characters}\hfill

For simplicity, we will write $\mathfrak{S}_+=\mathfrak{S}_{n_+}$ and $\mathfrak{S}_-=\mathfrak{S}_{n_-}$.
\subsubsection{}
Let $\chi$ be a $\sigma$-stable irreducible character of $\GL_n(q)$. According to \S \ref{sss-1st-lemma}, we may assume that $\chi$ is induced from some $\sigma$-stable standard Levi subgroup $M=M_0\times M_1$, and we can write $\chi$ as a linear combination of Deligne-Lusztig characters:
\begin{equation}\label{cptechi1}
\chi=|\mathfrak{S}_+\times\mathfrak{S}_-\times W_1|^{-1}\!\!\!\!\!\!\!\!\sum_{\substack{w=(w_+,w_-,w_1)\\\in\mathfrak{S}_+\times\mathfrak{S}_-\times W_1}}\!\!\!\!\!\!\!\!\varphi_+(w_+)\varphi_-(w_-)\varphi(w_1)R^{G}_{T_{w}}\theta_{w}.
\end{equation}
Suppose that $\chi$ is of type $\bs\alpha$. Fixing an isomorphism $W_1\cong\prod_i(\mathfrak{S}_{n_i}\times\mathfrak{S}_{n_i})$, the irreducible character $(\varphi_+,\varphi_-,\varphi)$ of $\mathfrak{S}_+\times\mathfrak{S}_-\times W_1$ corresponds to a unique ordered type, the unordered version of which is exactly $\{\bs\alpha\}$. In this subsection, $\bs\alpha$ will be regarded as an element of $\tilde{\mathfrak{T}}$. The sum over $w$ only depends on its conjugacy class, which we denote by $\mathcal{O}(w)$. Note that the conjugacy classes of $\mathfrak{S}_+\times\mathfrak{S}_-\times W_1$ are parametrised by a subset of $\tilde{\mathfrak{T}}$. The type of $\mathcal{O}(w)$ will be denoted by $\bs\tau$. We have 
\begin{equation}
\frac{|\mathcal{O}(w)|}{|\mathfrak{S}_+\times\mathfrak{S}_-\times W_1|}=\frac{1}{z_{\bs\tau}},\quad\varphi_+(w_+)\varphi_-(w_-)\varphi(w_1)=\chi^{\{\bs\alpha\}}_{\bs\tau}.
\end{equation}
We will therefore replace $\sum_{w}(-)\text{ by }\sum_{\mathcal{O}(w)}|\mathcal{O}(w)|\cdot(-)$. However, for each $\mathcal{O}(w)$, we will choose a $w$ representing it so that $T_{w}$ and $\theta_{w}$ have definite meanings.

If $s\in T^F$ is a semi-simple element, then the character formula (\ref{eq-char-formula}) reads
\begin{equation}\label{cptechi2}
R^{G}_{T_{w}}\theta_{w}(s)=\frac{1}{|C_G(s)^{F}|}\sum_{\{h\in G^{F}\mid hsh^{-1}\in T_{w}\}}Q^{C_G(s)}_{C_{h^{-1}T_{w}h}(s)}(1)\theta_{w}(hsh^{-1}).
\end{equation}
Denote by $\bs\mu\in\tilde{\mathfrak{T}}_s$ the type of the conjugacy class of $s$. The $C_{G}(s)^{F}$-conjugacy classes of the $F$-stable maximal tori of $C_{G}(s)$ are also parametrised by a subsets of $\tilde{\mathfrak{T}}$. If $C_{h^{-1}T_{w}h}(s)$ is of type $\bs\nu\in\tilde{\mathfrak{T}}$ (with $\nu_+=\nu_-=\varnothing$), it is known that $\mathcal{Q}^{\bs\mu}_{\bs\nu}(q)=Q^{C_G(s)}_{C_{h^{-1}T_{w}h}(s)}(1)$, where the left hand side is the Green polynomial defined by symmetric functions.

Let $A_{\tau}^F$ and $B_{\tau}$ be as in  \S \ref{A^F_tau}. Recall the notation $z_{\bs\nu}$ therein. Combining (\ref{cptechi1}) and (\ref{cptechi2}) gives
\begingroup
\allowdisplaybreaks
\begin{align}
\nonumber
\chi(s)=&\sum_{\bs\tau}\sum_{h\in A^F_{\tau}}\frac{1}{|C_G(s)^{F}|}\cdot\frac{\chi^{\{\bs\alpha\}}_{\bs\tau}Q^{C_G(s)}_{C_{h^{-1}T_{w}h}(s)}(1)}{z_{\bs\tau}}\theta_{w}(hsh^{-1})\\
%%%
\label{chi(s)}
\stackrel{\circled{1}}{=}&\sum_{\bs\tau}\sum_{\{\bs\nu\in B_{\tau}\}}\frac{1}{|C_G(s)^{F}|}\cdot\frac{\chi^{\{\bs\alpha\}}_{\bs\tau}\mathcal{Q}^{\bs\mu}_{\bs\nu}(q)}{z_{\bs\tau}}\sum_{h\in A^F_{\bs\tau,\bs\nu}}\theta_{w}(hsh^{-1})
\end{align}
\endgroup
Equality \circled{1} uses the surjective map $A^F_{\tau}\rightarrow B_{\tau}$ described in \S \ref{A^F_tau}.

\subsubsection{}
Now we can work in the context of \S \ref{Irr-reg-sig}. For each $j$, let $\bs\mu_j$ be the type of the semi-simple conjugacy class $C_j$.
\begin{Lem}\label{sum-Irr-sigma-w}
We have
$$\sum_{\chi\in\Irr^{\sigma}_{\bs\omega}}\prod^{k}_{j=1}\chi(C_{j})=\frac{K(\bs\omega_{\ast})}{N(\bs\omega_{\ast})}\prod^{k}_{j=1}\langle s_{\{\bs\omega\}}(\Z_j),\tilde{H}_{\bs\mu_j}(\X_j,q)\rangle,$$ where for each $j$, $\Z_j$ is an independent family of variables.
\end{Lem}
\begin{proof}
We calculate 
\begingroup
\allowdisplaybreaks
\begin{align}
\nonumber
%%%
&\sum_{\chi\in\Irr^{\sigma}_{\bs\omega}}\prod^{k}_{j=1}\chi(C_{j})\\
\nonumber
%%%
\stackrel{\circled{1}}{=}&\frac{1}{2^{l(\bs\omega_{\ast})}N(\bs\omega_{\ast})}\sum_{\bs\tau_1,\ldots,\bs\tau_{k}}\sum_{\substack{(\bs\nu_1,\ldots,\bs\nu_{k})\\\in\prod_jB_{\tau_j}}}\prod_{j=1}^{k}\frac{1}{|C_G(s)^{F}|}\cdot\frac{\chi^{\{\bs\omega\}}_{\bs\tau_j}\mathcal{Q}^{\bs\mu_j}_{\bs\nu_j}(q)}{z_{\bs\tau_j}}\!\!\!\!\sum_{\substack{(h_1,\ldots,h_{k})\\\in\prod_jA^F_{\bs\tau_j,\bs\nu_j}}}\sum_{\theta\in\Irr^{\sigma}_{reg}(M_1^F)}\prod_{j=1}^{k}\theta_j(h_js_jh_j^{-1})\\
\nonumber
%%%
\stackrel{\circled{2}}{=}&\frac{1}{2^{l(\bs\omega_{\ast})}N(\bs\omega_{\ast})}\sum_{\bs\tau_1,\ldots,\bs\tau_{k}}\sum_{\substack{(\bs\nu_1,\ldots,\bs\nu_{k})\\\in\prod_jB_{\tau_j}}}\prod_{j=1}^{k}\frac{1}{|C_G(s)^{F}|}\cdot\frac{\chi^{\{\bs\omega\}}_{\bs\tau_j}\mathcal{Q}^{\bs\mu_j}_{\bs\nu_j}(q)}{z_{\bs\tau_j}}\!\!\!\!\sum_{\substack{(h_1,\ldots,h_{k})\\\in\prod_jA^F_{\bs\tau_j,\bs\nu_j}}}(-1)^{l(\bs\omega_{\ast})}2^{l(\bs\omega_{\ast})}l(\bs\omega_{\ast})!\\
\nonumber
%%%
\stackrel{\circled{3}}{=}&\frac{K(\bs\omega_{\ast})}{N(\bs\omega_{\ast})}\prod^{k}_{j=1}\sum_{\bs\tau_j}\frac{z_{[\bs\tau_j]}\chi^{\{\bs\omega\}}_{\bs\tau_j}}{z_{\bs\tau_j}}\sum_{\{\bs\nu_j\mid[\bs\nu_j]=[\bs\tau_j]\}}\frac{\mathcal{Q}^{\bs\mu_j}_{\bs\nu_j}(q)}{z_{\bs\nu_j}}\\
%%%
\stackrel{\circled{4}}{=}&\frac{K(\bs\omega_{\ast})}{N(\bs\omega_{\ast})}\prod^{k}_{j=1}\langle s_{\{\bs\omega\}}(\Z_j),\tilde{H}_{\bs\mu_j}(\Z_j,q)\rangle,
\end{align}
\endgroup
In equality \circled{1}, we have applied (\ref{chi(s)}) to each $\chi(C_j)$. The factor $2^{l(\bs\omega_{\ast})}N(\bs\omega_{\ast})$ comes from the surjective map (\ref{eq-reg->omega}). In equality \circled{2}, we have used Lemma \ref{sumofreg}. We have \circled{3} because the summation over $h_j$ is independent of $h_j$, which produces $|C_G(s)^{F}|z_{[\bs\tau_j]}/z_{\bs\nu_j}$ by Proposition \ref{card-Ataunu}. Equality \circled{4} uses Lemma \ref{<s,H>}.
\end{proof}

\subsection{Some notations}
\subsubsection{}
For any partition $\lambda$, the Hook polynomial $H_{\lambda}(q)$ is defined by:
\begin{equation}\label{eq-hook}
H_{\lambda}(q):=\prod_{x\in\lambda}(1-q^{h(x)}),
\end{equation}
where $x$ runs over the boxes in the Young diagram of $\lambda$, and $h(x)$ is the hook length of $x$. If we denote by $\lambda^{\ast}$ the dual partition of $\lambda$, then 
\begin{equation}\label{eq-hook-lgth}
\sum_{x\in\lambda}h(x)=|\lambda|+n(\lambda)+n(\lambda^{\ast}).
\end{equation}
For any $\bs\omega=\omega_+\omega_-(\omega_i)\in\mathfrak{T}$, put $H_{\bs\omega}(q):=H_{\omega_+}(q)H_{\omega_-}(q)\prod_iH_{\omega_i}(q)$.  Then for $\chi\in\Irr_{\bs\omega}^{\sigma}$, we have (\cite[Chapter IV, \S 6 (6.7)]{Mac}):
\begin{equation}\label{eq-Mac6.7}
\frac{|\GL_n(q)|}{\chi(1)}=(-1)^nq^{\frac{1}{2}n(n-1)-n(\{\bs\omega\})}H_{\{\bs\omega\}}(q).
\end{equation}
Recall the notation $\{\bs\omega\}$ in \S \ref{sss-types}.

\subsubsection{}\label{dim-Ch}
According to \cite[Theorem 4.6]{Shu1}, the dimension of the character variety is given by:
\begingroup
\allowdisplaybreaks
\begin{align*}
d:=&(2g-2)\dim G+\sum^{k}_{j=1}\dim C_j\\
=&(2g-2)\dim G+\sum^{k}_{j=1}(\dim G-\dim C_G(s_j))\\
=&(2g-2)n^2+kn^2-\sum^{k}_{j=1}(2n(\bs\mu_j)+n)\\
=&n^2(2g+k-2)-kn-\sum^{k}_{j=1}2n(\bs\mu_j).
\end{align*}
\endgroup
Note that $n((1^n))$ is equal to the number of positive roots in $\GL_n$.

\subsection{The Formula for E-Polynomials}
\subsubsection{}\label{Omega(q)}
Define the infinite series:
\begingroup
\allowdisplaybreaks
\begin{align*}
\Omega_{\bullet}(q):=&\sum_{\lambda\in\mathcal{P}}q^{|\lambda|(1-g)}(q^{-n(\lambda)}H_{\lambda}(q))^{2g+k-2}\prod^k_{j=1}s_{\lambda}[\frac{\Z_j}{1-q}]\\
\Omega_{\ast}(q):=&\sum_{\lambda\in\mathcal{P}}q^{|\lambda|(2-2g)}(q^{-n(\lambda)}H_{\lambda}(q))^{4g+2k-4}\prod^k_{j=1}s_{\lambda}[\frac{\Z_j}{1-q}]^2,
\end{align*}
\endgroup

Recall that for each $1\le j\le k$, we denote by $\bs\mu_j$ the type of the semi-simple conjugacy class $C_j$.
\begin{Thm}\label{Main-Thm}
Suppose that $\mathcal{C}$ is generic. Then with the notations above, we have
\begingroup
\allowdisplaybreaks
\begin{align}
|\Ch_{\mathcal{C}}(\mathbb{F}_q)|=&~q^{\frac{1}{2}d}
\left\langle\frac{\Omega_{\bullet}(q)^2}{\Omega_{\ast}(q)},\prod^{k}_{j=1}h_{\bs\mu^{\ast}_j}(\Z_j)\right\rangle.
\end{align}
\endgroup
\end{Thm}
\begin{proof}
Proposition \ref{M=U/G}, equation (\ref{RepCFq}) and Proposition \ref{Frob-Form} now give
\begingroup
\allowdisplaybreaks
\begin{align}
\nonumber
|\Ch_{\mathcal{C}}(\mathbb{F}_q)|=&\sum_{\chi\in\Irr(G^F)^{\sigma}}\left(\frac{|G^F|}{\chi(1)}\right)^{2g-2}\prod_{j=1}^{k}\frac{|C_{j}|\chi(C_{j})}{\chi(1)}\\%%%
=&\sum_{\substack{\bs\omega\in\mathfrak{T}\\|\{\bs\omega\}|=n}}\frac{|G^F|^{2g-2}\prod^{k}_{j=1}|C_{j}|}{\chi(1)^{2g+k-2}}\sum_{\chi\in\Irr^{\sigma}_{\bs\omega}}\prod^{k}_{j=1}\chi(C_{j}).
\label{eq-brute-E-poly}
\end{align}
\endgroup
In the second equality, we have used the fact that irreducible characters of the same type  have the same degree (see (\ref{eq-Mac6.7})). This formula holds for any generic tuple of conjugacy classes. 

We calculate
\begingroup
\allowdisplaybreaks
\begin{align}
\nonumber%%
&|\Ch_{\mathcal{C}}(\mathbb{F}_q)|\\
\nonumber%%
=&\sum_{\substack{\bs\omega\in\mathfrak{T}\\|\{\bs\omega\}|=n}}\frac{|G^F|^{2g-2}\prod^{k}_{j=1}|C_{j}|}{\chi(1)^{2g+k-2}}\sum_{\chi\in\Irr^{\sigma}_{\bs\omega}}\prod^{k}_{j=1}\chi(C_{j})\\%%
\nonumber
\stackrel{\circled{1}}{=}&\sum_{\bs\omega\in\mathfrak{T}}\frac{K(\bs\omega_{\ast})}{N(\bs\omega_{\ast})}\frac{|G^F|^{2g-2}\prod^{k}_{j=1}|C_{j}|}{\chi(1)^{2g+k-2}}\prod^{k}_{j=1}\langle s_{\{\bs\omega\}}(\Z_j),\tilde{H}_{\bs\mu_j}(\Z_j,q)\rangle\\
%%%
\nonumber
=&\sum_{\bs\omega\in\mathfrak{T}}\frac{K(\bs\omega_{\ast})}{N(\bs\omega_{\ast})}\left(\frac{|G^F|}{\chi(1)}\right)^{2g+k-2}\prod^{k}_{j=1}\frac{|C_{j}|}{|G^F|}\prod^{k}_{j=1}\langle s_{\{\bs\omega\}}(\Z_j),\tilde{H}_{\bs\mu_j}(\Z_j,q)\rangle\\
%%%
\nonumber
\stackrel{\circled{2}}{=}&\sum_{\bs\omega\in\mathfrak{T}}\frac{K(\bs\omega_{\ast})}{N(\bs\omega_{\ast})}((-1)^nH_{\{\bs\omega\}}(q)q^{\frac{1}{2}n(n-1)-n(\{\bs\omega\})})^{2g+k-2}\prod^{k}_{j=1}\langle s_{\{\bs\omega\}}(\Z_j),\frac{|C_{j}|}{|G^F|}\tilde{H}_{\bs\mu_j}(\Z_j,q)\rangle\\%%
\nonumber
=&(-1)^{kn}q^{\frac{1}{2}(n^2(2g+k-2)-kn)}\\
&\cdot\left\langle \sum_{\bs\omega\in\mathfrak{T}}\!\frac{K(\bs\omega_{\ast})}{N(\bs\omega_{\ast})}q^{(1-g)|\{\bs\omega\}|}(H_{\{\bs\omega\}}(q)q^{-n(\{\bs\omega\})})^{2g+k-2}\!\prod^{k}_{j=1}\!s_{\{\bs\omega\}}(\Z_j),\prod^{k}_{j=1}\!\!\frac{|C_{j}|}{|G^F|}\tilde{H}_{\bs\mu_j}(\Z_j,q)\rangle \right\rangle.\label{innerprod}
\end{align}
\endgroup
In \circled{1} the sum can be taken over the entire $\mathfrak{T}$ because if $\bs\omega$ was not of size $N$ then the inner product of symmetric functions would vanish. We have also used Lemma \ref{sum-Irr-sigma-w} in this equality. Equality \circled{2} uses (\ref{eq-Mac6.7}).

Equation (\ref{1/|GLn|}) shows that
\begin{equation}
\frac{|C_{j}|}{|G^F|}=q^{-2n(\bs\mu_j)-|\bs\mu_j|}b_{\bs\mu_j}(q^{-1})^{-1}.
\end{equation}
Combined with Lemma \ref{XoXlBCType}, this shows
\begin{equation}
q^{-n(\bs\mu_j)}h_{\bs\mu^{\ast}_j}[\frac{\Z_j}{1-q}]=
(-1)^{|\bs\mu_j|}\frac{|C_{j}|}{|G^F|}\tilde{H}_{\bs\mu_j}(\Z_j;q).
\end{equation}

For any symmetric functions $u(\Z)$ and $v(\Z)$, we have $$\left\langle u[\frac{\Z}{1-q}],v(\Z)\right\rangle=\left\langle u(\Z),v[\frac{\Z}{1-q}]\right\rangle.$$ This can be checked on the basis of power sums. We can then move $1/(1-q)$ to the left hand side of the inner product.

Rewrite the left hand side of the inner product (\ref{innerprod}) as follows:
\begingroup
\allowdisplaybreaks
\begin{align}
\nonumber%%
&\sum_{\bs\omega\in\mathfrak{T}}\frac{K(\bs\omega_{\ast})}{N(\bs\omega_{\ast})}q^{(1-g)|\{\bs\omega\}|}(H_{\{\bs\omega\}}(q)q^{-n(\{\bs\omega\})})^{2g+k-2}\prod^{k}_{j=1}s_{\{\bs\omega\}}[\frac{\Z_j}{1-q}]\\
\nonumber%%
=&~~\left(\sum_{\omega_+\in\mathcal{P}}q^{(1-g)|\omega_+|}(H_{\omega_+}(q)q^{-n(\omega_+)})^{2g+k-2}\prod^{k}_{j=1}s_{\omega_+}[\frac{\Z_j}{1-q}]\right)\\
\nonumber%%
&\cdot\left(\sum_{\omega_-\in\mathcal{P}}q^{(1-g)|\omega_-|}(H_{\omega_-}(q)q^{-n(\omega_-)})^{2g+k-2}\prod^{k}_{j=1}s_{\omega_-}[\frac{\Z_j}{1-q}]\right)\\
\nonumber%%
&\cdot\left(\sum_{\bs\omega_{\ast}}\frac{K(\bs\omega_{\ast})}{N(\bs\omega_{\ast})}q^{(2-2g)|\bs\omega_{\ast}|}(H_{\bs\omega_{\ast}}(q)^2q^{-2n(\bs\omega_{\ast})})^{2g+k-2}\prod^{k}_{j=1}s_{\bs\omega_{\ast}}[\frac{\Z_j}{1-q}]^2\right).
\end{align}
\endgroup

By definition
\begingroup
\allowdisplaybreaks
\begin{align}
\nonumber%%
\frac{1}{\Omega_{\ast}(q)}=&\sum_{m\ge 0}(-1)^m\!\!\!\!\sum_{\substack{\bs\omega_{\ast}=(m_{\lambda})_{\lambda}\in\mathfrak{T}_{\ast}\\l(\bs\omega_{\ast})=m}}\frac{m!}{\prod_{\lambda}m_{\lambda}!}q^{(2-2g)|\bs\omega_{\ast}|}(H_{\bs\omega_{\ast}}(q)^2q^{-2n(\bs\omega_{\ast})})^{2g+k-2}\prod^{k}_{j=1}s_{\bs\omega_{\ast}}[\frac{\Z_j}{1-q}]^2\\
%%%
=&\sum_{\bs\omega_{\ast}\in\mathfrak{T}_{\ast}}\frac{K(\bs\omega_{\ast})}{N(\bs\omega_{\ast})}q^{(2-2g)|\bs\omega_{\ast}|}(H_{\bs\omega_{\ast}}(q)^2q^{-2n(\bs\omega_{\ast})})^{2g+k-2}\prod^{k}_{j=1}s_{\bs\omega_{\ast}}[\frac{\Z_j}{1-q}]^2,
\end{align}
\endgroup
whence
\begingroup
\allowdisplaybreaks
\begin{align*}
\sum_{\bs\omega\in\mathfrak{T}}\frac{K(\bs\omega_{\ast})}{N(\bs\omega_{\ast})}q^{(1-g)|\{\bs\omega\}|}(H_{\{\bs\omega\}}(q)q^{-n(\{\bs\omega\})})^{2g+k-2}\prod^{k}_{j=1}s_{\{\bs\omega\}}[\frac{\Z_j}{1-q}]^2
=\frac{\Omega_{\bullet}(q)^2}{\Omega_{\ast}(q)}
\end{align*}
\endgroup

Therefore,
\begin{equation}
|\Ch_{\mathcal{C}}(\mathbb{F}_q)|
=q^{\frac{1}{2}\left(n^2(2g+k-2)-kn\right)-\sum^{k}_{j=1}n(\bs\mu_j)}\cdot\left\langle\frac{\Omega_{\bullet}(q)^2}{\Omega_{\ast}(q)},\prod^{k}_{j=1}h_{\bs\mu^{\ast}_j}(\Z_j)\right\rangle.
\end{equation}
Finally, we use \S \ref{dim-Ch}.
\end{proof}

\begin{Cor}
(i). The $t=-1$ specialisation of part (iii) of Conjecture \ref{The-Conj} holds. (ii). The $t=-1$ specialisation of Conjecture \ref{Cur-Poin} holds.
\end{Cor}
\begin{proof}
Specialising $\Omega_{\bullet}(z,w)$ and $\Omega_{\ast}(z,w)$ to $z=\sqrt{q}$, $w=\sqrt{q}^{-1}$, we get
\begingroup
\allowdisplaybreaks
\begin{align*}
\Omega_{\bullet}(\sqrt{q},\sqrt{q}^{-1})=&\sum_{\lambda\in\mathcal{P}}N_{\lambda}(q,q^{-1})^{g-1}\prod^k_{j=1}H_{\lambda}(\Z_j;q,q^{-1})\\
\Omega_{\ast}(\sqrt{q},\sqrt{q}^{-1})=&\sum_{\lambda\in\mathcal{P}}N_{\lambda}(q,q^{-1})^{2g-2}\prod^k_{j=1}H_{\lambda}(\Z_j;q,q^{-1})^2.
\end{align*}
\endgroup
Then use \cite[Lemma 8.1.3, Lemma 8.1.4]{Shu3} and \cite[Proposition 3.3.2]{Hai}. This proves (i).

By \cite[Remark 8.1.2]{Shu3}, we have
$$
\Omega_{\bullet}(z,w)=\Omega_{\bullet}(w,z),\quad\Omega_{\ast}(z,w)=\Omega_{\ast}(w,z),
$$
and so
$$
\mathbb{H}_{\bs\mu}(z,w)=\mathbb{H}_{\bs\mu}(w,z),\quad\mathbb{H}_{\bs\mu}(z,w)=\mathbb{H}_{\bs\mu}(-z,-w).
$$
We deduce from this that part (iii) of Conjecture \ref{The-Conj} implies Conjecture \ref{Cur-Poin}. Therefore the $t=-1$ specialisation holds. 
\end{proof}

%%%%%%%%%%%%%%%%%%%%%%
\section{Mixed Hodge Polynomial}\label{Sec-MHP}
\subsection{Evidences of the conjecture}\hfill

We check that our conjectural formula for the mixed Hodge polynomial is consistent with known cohomological results when $n=1$ and $n=2$. For simplicity, we assume $k=1$.
\subsubsection{}
We first consider the case $n=1$. The definition of the character variety reads
$$
\Ch_{\mathcal{C}}(\Sigma)=\left\{(A_i,B_i)_i\in \GL_1^{2g}\mid B_1^{-2}X_1=1\right\}\ds\GL_1,
$$
for a given element $X_1$ of $\GL_1$. The action of $\GL_1$ on $A_1$ is given by 
$$gA_1\sigma(g)^{-1}=g^2A_1,\text{ for $g\in\GL_1$,}$$
and its action on all other components is trivial. We see that there are two connected components corresponding to the two solutions of $B_1$, and each connected component is a $2g-2$-dimensional torus. The mixed Hodge polynomial is then $2(t+qt^2)^{2g-2}$.

The formula in Conjecture \ref{The-Conj} (iii) gives 
$$2(t\sqrt{q})^{2g-2}N_{(1)}(zw,z^2,w^2)^{2g-2}\mid_{z=-t\sqrt{q},w=\sqrt{q}^{-1}}=2(t\sqrt{q})^{2g-2}(t\sqrt{q}+\frac{1}{\sqrt{q}})^{2g-2},$$
as expected.

\subsubsection{}
Then we consider the case $n=2$. In this case, the centraliser of $\sigma$ in $\GL_2$ is just $\SL_2$, and $\sigma$ acts on the center by inversion. Now the character variety is defined as
$$
\Ch_{\mathcal{C}}(\Sigma):=\left\{(A_i,B_i)_i(X_1)\in \GL_2^{2g}\times C_1\mid A_1\sigma(B_1)A_1^{-1}B_1^{-1}\prod_{i=2}^g[A_i,B_i]X_1=1\right\}\ds\GL_2,
$$
where $C_1$ is the conjugacy class of $\diag(a,a^{-1})$ with $a^2\ne 1$. For any $1\le i\le g$, write $A_i=a_i\alpha_i$ and $B_i=b_i\beta_i$, with $a_i$, $b_i\in\SL_2$ and $\alpha_i$, $\beta_i\in\mathbb{C}^{\ast}$, the center of $\GL_2$. Thus we realise $\Ch_{\mathcal{C}}$ as a quotient of 
$$
M:=\left\{(a_i,b_i)_i(\alpha_i,\beta_i)_i(X_1)\in \SL_2^{2g}\times(\mathbb{C}^{\ast})^{2g}\times C_1\mid \prod_{i=1}^g[a_i,b_i]\beta_1^{-2}X_1=1\right\}\ds\GL_2,
$$
where $\GL_2$ acts on $X_1$, $a_i$ and $b_i$, $1\le i\le g$, by conjugation, acts trivially on $\beta_1$, $\alpha_i$, $\beta_i$, $2\le i\le g$, and acts on $\alpha_1$ by 
$$g:\alpha_1\mapsto\det(g)\alpha_1,\text{ for $g\in\GL_2$.}$$
The finite group $\Gamma:=\{\pm\Id\}^{2g}\subset(\mathbb{C}^{\ast})^{2g}$ acts on $M$ in such a way that $\Ch_{\mathcal{C}}\cong M/\Gamma$.
Define
$$
M_{\pm}:=\left\{(a_i,b_i)_i(X_1)\in \SL_2^{2g}\times C_1\mid \prod_{i=1}^g[a_i,b_i]X_1=\pm1\right\}\ds\SL_2,
$$
where $\SL_2$ acts by conjugation on each component. Since it is necessary that $\beta_1^{-2}=\pm\Id$, we have 
$$M=M_+\times\{\pm1\}\times(\mathbb{C}^{\ast})^{2g-2}\sqcup M_-\times\{\pm\sqrt{-1}\}\times(\mathbb{C}^{\ast})^{2g-2},$$
where $\{\pm1\}$ and $\{\pm\sqrt{-1}\}$ are the solutions of $\beta_1^2=\pm\Id$. The induced $\Gamma$-action on $M_+\times\{\pm1\}\times(\mathbb{C}^{\ast})^{2g-2}$ is as follows. Let $(e_i)_{1\le i\le 2g}\in\Gamma$, then $e_2$ identifies the two components:
$$e_2:M_+\times\{+1\}\times(\mathbb{C}^{\ast})^{2g-2}\lisom M_+\times\{-1\}\times(\mathbb{C}^{\ast})^{2g-2},$$while $e_i$, $i\ne 2$, acts in the usual manner. Denote by $\Gamma_0$ the subgroup of $\Gamma$ consisting of elements whose $e_2$-component is trivial. We see that 
$$
(M_+\times\{\pm1\}\times(\mathbb{C}^{\ast})^{2g-2})/\Gamma\cong (M_+\times(\mathbb{C}^{\ast})^{2g-2})/\Gamma_0.
$$The same argument applies to $M_-$. We see that
$$
\Ch_{\mathcal{C}}\cong (M_+\times(\mathbb{C}^{\ast})^{2g-2})/\Gamma_0\sqcup (M_-\times(\mathbb{C}^{\ast})^{2g-2})/\Gamma_0.
$$

Considering the cohomology with compact support, we have 
$$
H_c((M_+\times(\mathbb{C}^{\ast})^{2g-2})/\Gamma_0)\cong H_c(M_+)^{\Gamma_0}\otimes H_c((\mathbb{C}^{\ast})^{2g-2}),
$$
since the action of $\Gamma_0$ on the cohomology of $(\mathbb{C}^{\ast})^{2g-2}$ is trivial. Denote by $\kappa$ the irreducible character of $\Gamma$ that is non trivial only on the $e_2$-component. Then 
$$
H_c(M_+)^{\Gamma_0}=H_c(M_+)_{st}\oplus H_c(M_+)_{\kappa},
$$
where $H_c(M_+)_{st}=H_c(M_+)^{\Gamma}$ is the $\Gamma$-invariant subspace and $H_c(M_+)_{\kappa}$ is the $\kappa$-isotypic component. This is also a direct sum of mixed Hodge structures. The mixed Hodge polynomial of $H_c(M_+)_{st}$ can be deduced from the mixed Hodge polynomial of the $\GL_2$-character variety defined by the conjugacy class $C_1$. The mixed Hodge polynomial of $H_c(M_+)_{\kappa}$ is equal to the mixed Hodge polynomial of 
$$
H_c(M_+)_{var}:=\bigoplus_{\substack{\gamma\in\Irr(\Gamma)\\\gamma\ne 1}}H_c(M_+)_{\gamma}
$$
divided by $(2^{2g}-1)$, since the mapping class group acts transitively on the set of non trivial irreducible characters of $\Gamma$. We call $H_c(M_+)_{var}$ the variant part of the cohomology.

\subsubsection{}
According to \cite[\S 4.4]{BY}, the Poincar\'e polynomial of the $\SL_2$-character variety with one regular semisimple conjugacy class is given by 
\begin{equation}
\frac{(1+t^3)^{2g}+t^{4g-1}(1+t)^{2g}\left((2g-1)t-2g\right)}{(t^2-1)^2}+(2^{2g}-1)t^{4g-2}(t+1)^{2g-2}.
\end{equation}
We deduce from this expression the Poincar\'e polynomial of the cohomology with compact support by Poincar\'e duality. The variant part is
\begin{equation}\label{var-SL2-P}
(2^{2g}-1)t^{6g-4}(t+1)^{2g-2}.
\end{equation}
On the other hand, using the character table of $\SL_2(q)$ in \cite[\S 15.9]{DM91}, we compute the E-polynomial of this character variety. The result is 
\begin{equation}
(q-1)^{2g-2}\left((2^{2g}-2)q^{2g-1}+q^{2g-1}(q+1)^{2g-1}+(q+1)^{2g-1}\right).
\end{equation}
The stable part of the E-polynomial can be computed using the character table of $\GL_2(q)$ in \cite[\S 15.9]{DM91}, and the result is
\begin{equation}
(q-1)^{2g-2}\left(-q^{2g-1}+q^{2g-1}(q+1)^{2g-1}+(q+1)^{2g-1}\right).
\end{equation}
We deduce the variant part of the E-polynomial:
\begin{equation}\label{var-SL2-E}
(2^{2g}-1)q^{2g-1}(q-1)^{2g-2}.
\end{equation}

The variant part of the Poincar\'e polynomial and the variant part of the E-polynomial have an obvious common deformation:
\begin{equation}\label{var-MHP}
(2^{2g-1}-1)t^{6g-4}q^{2g-1}(qt+1)^{2g-2}.
\end{equation}
If we write $z=-t\sqrt{q}$ and $w=1/\sqrt{q}$, then this is just $(2^{2g-1}-1)(t\sqrt{q})^{6g-4}(z-w)^{2g-2}$, with $6g-4$ equal to the dimension of the $\SL_2$-character variety.
\begin{Rem}
In principle, we can use the method of \cite[\S 4.4]{dCHM} to prove that each cohomological degree has only one weight, and then the variant part of the mixed Hodge polynomial is given by (\ref{var-MHP}). Since our goal is to check that the conjectural formula is consistent with known results, we omit the proof.
\end{Rem}

The stable part of the mixed Hodge polynomial can be computed using the conjectural formula of Hausel-Letellier-Rodriguez-Villegas. Taking into account the variant part, we see that the mixde Hodge polynomial of $\Ch_{\mathcal{C}}$ is given by
\begingroup
\allowdisplaybreaks
\begin{align*}
&2\frac{(z^3-w)^{2g}(z-w)^{2g-2}}{(z^4-1)(z^2-w^2)}(1+z^2)+2\frac{(z-w^3)^{2g}(z-w)^{2g-2}}{(z^2-w^2)(1-w^4)}(1+w^2)\\&-2\frac{(z-w)^{4g-2}}{(z^2-1)(1-w^2)}+2(z-w)^{4g-4}
\end{align*}
\endgroup
with $z=-t\sqrt{q}$ and $w=1/\sqrt{q}$, and multiplied by $(t\sqrt{q})^{6g-4}$. This is exactly the formula predicted by Conjecture \ref{The-Conj}.

\addtocontents{toc}{\protect\setcounter{tocdepth}{-1}}
\bibliographystyle{alpha}
\bibliography{BIB}

\begin{thebibliography}{dCHM12}

\bibitem[BF14]{BF}
Roman Bezrukavnikov and Michael Finkelberg.
\newblock Wreath {M}acdonald polynomials and the categorical {M}c{K}ay
  correspondence.
\newblock {\em Camb. J. Math.}, 2(2):163--190, 2014.
\newblock With an appendix by Vadim Vologodsky.

\bibitem[BY96]{BY}
Hans~U. Boden and K\^{o}ji Yokogawa.
\newblock Moduli spaces of parabolic {H}iggs bundles and parabolic {$K(D)$}
  pairs over smooth curves. {I}.
\newblock {\em Internat. J. Math.}, 7(5):573--598, 1996.

\bibitem[dCHM12]{dCHM}
Mark Andrea~A. de~Cataldo, Tam\'{a}s Hausel, and Luca Migliorini.
\newblock Topology of {H}itchin systems and {H}odge theory of character
  varieties: the case {$A_1$}.
\newblock {\em Ann. of Math. (2)}, 175(3):1329--1407, 2012.

\bibitem[DM91]{DM91}
Fran\c{c}ois Digne and Jean Michel.
\newblock {\em Representations of finite groups of {L}ie type}, volume~21 of
  {\em London Mathematical Society Student Texts}.
\newblock Cambridge University Press, Cambridge, 1991.

\bibitem[Got94]{Got}
Peter~B. Gothen.
\newblock The {B}etti numbers of the moduli space of stable rank {$3$} {H}iggs
  bundles on a {R}iemann surface.
\newblock {\em Internat. J. Math.}, 5(6):861--875, 1994.

\bibitem[Hai03]{Hai}
Mark Haiman.
\newblock Combinatorics, symmetric functions, and {H}ilbert schemes.
\newblock In {\em Current developments in mathematics, 2002}, pages 39--111.
  Int. Press, Somerville, MA, 2003.

\bibitem[Hit87]{Hit}
N.~J. Hitchin.
\newblock The self-duality equations on a {R}iemann surface.
\newblock {\em Proc. London Math. Soc. (3)}, 55(1):59--126, 1987.

\bibitem[HLRV11]{HLR}
Tam\'{a}s Hausel, Emmanuel Letellier, and Fernando Rodriguez-Villegas.
\newblock Arithmetic harmonic analysis on character and quiver varieties.
\newblock {\em Duke Math. J.}, 160(2):323--400, 2011.

\bibitem[HRV08]{HRV}
Tam\'{a}s Hausel and Fernando Rodriguez-Villegas.
\newblock Mixed {H}odge polynomials of character varieties.
\newblock {\em Invent. Math.}, 174(3):555--624, 2008.
\newblock With an appendix by Nicholas M. Katz.

\bibitem[LN08]{LN}
G\'{e}rard Laumon and Bao~Ch\^{a}u Ng\^{o}.
\newblock Le lemme fondamental pour les groupes unitaires.
\newblock {\em Ann. of Math. (2)}, 168(2):477--573, 2008.

\bibitem[LS77]{LS}
George Lusztig and Bhama Srinivasan.
\newblock The characters of the finite unitary groups.
\newblock {\em J. Algebra}, 49(1):167--171, 1977.

\bibitem[Mac95]{Mac}
I.~G. Macdonald.
\newblock {\em Symmetric functions and {H}all polynomials}.
\newblock Oxford Mathematical Monographs. The Clarendon Press, Oxford
  University Press, New York, second edition, 1995.
\newblock With contributions by A. Zelevinsky, Oxford Science Publications.

\bibitem[Mel18]{Me1}
Anton Mellit.
\newblock Integrality of {H}ausel-{L}etellier-{V}illegas kernels.
\newblock {\em Duke Math. J.}, 167(17):3171--3205, 2018.

\bibitem[Mel19]{Me2}
Anton Mellit.
\newblock Cell decompositions of character varieties, 2019.

\bibitem[Mel20a]{Me4}
Anton Mellit.
\newblock Poincar\'{e} polynomials of character varieties, {M}acdonald
  polynomials and affine {S}pringer fibers.
\newblock {\em Ann. of Math. (2)}, 192(1):165--228, 2020.

\bibitem[Mel20b]{Me3}
Anton Mellit.
\newblock Poincar\'{e} polynomials of moduli spaces of {H}iggs bundles and
  character varieties (no punctures).
\newblock {\em Invent. Math.}, 221(1):301--327, 2020.

\bibitem[Sch16]{Sch}
Olivier Schiffmann.
\newblock Indecomposable vector bundles and stable {H}iggs bundles over smooth
  projective curves.
\newblock {\em Ann. of Math. (2)}, 183(1):297--362, 2016.

\bibitem[Ses77]{Ses77}
C.~S. Seshadri.
\newblock Geometric reductivity over arbitrary base.
\newblock {\em Advances in Math.}, 26(3):225--274, 1977.

\bibitem[Shu20a]{Shu2}
Cheng Shu.
\newblock The character table of $\text{GL}_n(q)\rtimes\lb\sigma\rb$.
\newblock {\em arXiv}, 2020.

\bibitem[Shu20b]{Shu3}
Cheng Shu.
\newblock E-polynomials of generic
  $\mathbf{\GL_n\rtimes\lb\sigma\rb}~$-character varieties: Branched case.
\newblock {\em arXiv}, 2020.

\bibitem[Shu20c]{Shu1}
Cheng Shu.
\newblock On character varieties with non connected structure groups.
\newblock {\em arXiv}, 2020.

\end{thebibliography}
\end{document}